\documentclass[11pt,a4paper]{article}
\usepackage{amsmath,amsthm,amssymb,graphicx,hyperref}

\renewcommand{\phi}{\varphi}

\theoremstyle{plain}
\newtheorem{theorem}{Theorem}
\newtheorem{lemma}{Lemma}
\newtheorem{proposition}{Proposition}

\theoremstyle{definition}

\theoremstyle{remark}
\newtheorem{remark}{Remark}

\newcommand{\eqn}[2]{\begin{equation}\label{#1}#2\end{equation}}
\newcommand{\eqnst}[1]{\begin{equation*}#1\end{equation*}}
\newcommand{\eqnspl}[2]{\begin{equation}\begin{split}\label{#1}%
    #2\end{split}\end{equation}}
\newcommand{\eqnsplst}[1]{\begin{equation*}\begin{split}%
    #1\end{split}\end{equation*}}

\newcommand{\C}{\mathbb C}

\newcommand{\R}{\mathbb R}

\newcommand{\E}{\mathbf{E}}

\def\caR{\mathcal{R}}

\def\bx{\mathbf{x}}

\def\ff{\mathfrak{f}}
\def\fg{\mathfrak{g}}
\def\nf{\mathfrak{n}}

\def\Re{\mathfrak{Re}}

\def\eps{\varepsilon}

\DeclareMathOperator*{\argmin}{arg\,min}

\def\Unif{{\mathsf{Unif}}}

\def\sgn{{\mathrm{sgn}}}

\def\tc{\widetilde{c}}

\def\tchi{\widetilde{\chi}}
\def\tkappa{\widetilde{\kappa}}
\def\tnf{\widetilde{\nf}}
\def\os{\overline{s}}

\def\oS{\overline{S}}

\def\tY{\widetilde{Y}}
\def\prob{\mathbf{P}}

\def\Var{\mathsf{Var}}

\begin{document}

\title{{\bf Edgeworth expansion with error estimates for power law shot noise}}

\author{
Antal A.~J\'arai
\footnote{Department of Mathematical Sciences,
University of Bath,
Claverton Down, Bath, BA2 7AY, United Kingdom, 
Email: {\tt A.Jarai@bath.ac.uk}}%
}

\maketitle

\footnotesize
\begin{quote}
{\bf Abstract}: 
Consider a homogeneous Poisson process in $\R^d$, $d \ge 1$.
Let $R_1 < R_2 < \dots$ be the distances of the points from the origin, 
and let $S = R_1^{-\gamma} + R_2^{-\gamma} + \dots$, where $\gamma > d$
is a parameter. Let $\oS^{(r)} = \sum_k R_k^{-\gamma} \, \mathbf{1}_{R_k \ge r}$
be the contribution to $S$ outside radius $r$. 
For large enough $r$, and any $\os$ in the support of $\oS^{(r)}$, 
consider the change of measure that shifts the mean to $\os$.  
We derive rigorous error estimates for the Edgeworth expansion of the
transformed random variable.
Our error terms are uniform in $\os$, 
and we give explicitly the dependence of
the error on $r$ and the order $k$ of the expansion. 
As an application, we provide a scheme that approximates the conditional distribution
of $R_1$ given $S = s$ to any desired accuracy, with error bounds that are uniform in $s$. 
Along the way, we prove a stochastic comparison between 
$(R_1, R_2, \dots)$ given $S = s$ and unconditioned radii 
$(R'_1, R'_2, \dots)$.
\end{quote}
\normalsize

{\bf Key-words}: Poisson point process, positive stable law,
Edgeworth expansion, change of measure, power law shot noise, pathloss

{\bf MSC 2010:} 60F05

\vspace{12pt}


\section{Introduction}
\label{sec:intro}

Consider a homogeneous Poisson point process in $\R^d$.
Let $0 < R_1 < R_2 < \dots$ be the distances of the points from the origin
in increasing order, and let 
\eqn{e:signal}
{ S
  = \sum_{k=1}^\infty R_k^{-\gamma} \qquad\qquad
  \oS^{(r)} 
  = \sum_{k : R_k > r} R_k^{-\gamma}, \quad r \ge 0,}
where $\gamma > d$ is a parameter. 
It is well known, and easy to see 
that $S = \oS^{(0)}$ is a stable random variable of index $d/\gamma$ \cite[Section 1.4]{STbook}.

\subsection{Motivation}

The setup above has the following interpretation 
relevant to wireless communication \cite[Chapter 5]{Haenggi}, \cite{WPS09}. 
Assume that $d = 2$ and at each point of the 
Poisson process a radio transmitter is located that emits a signal at 
unit power. The signal experiences a path loss $r^{-\gamma}$ at distance 
$r$ from the source, and hence an observer located at the origin receives
a total signal $S$. Suppose that the observer measures $S$ and is interested in
the distance $R_1$ to the nearest transmitter (providing the strongest signal
for the observer). The density of $R_1$, given $S = s$, is  
\eqn{e:Bayes}
{ f_{R_1 \,|\, S} (r_1 \,|\, s)
  = \frac{1}{f_S(s)} f_{R_1}(r_1) \, f_{S \,|\, R_1} (s \,|\, r_1)
  = \frac{1}{f_S(s)} f_{R_1}(r_1) \, f_{\oS^{(r_1)}}(s - r_1^{-\gamma}). }
As no closed formula is known for the density of $\oS^{(r)}$, it is of interest
to find approximations. The main technical result of this paper gives rigorous error 
estimates for an approximation of this density when $r$ is large. 
This large $r$ result can be used, via a modification of \eqref{e:Bayes} described
further below, to approximate $f_{R_1 \,|\, S}$ to any desired accuracy.

\subsection{Approximating $\oS^{(r)}$}

Let us first describe the approximation to $\oS^{(r)}$ we consider.
We let $\mu(r) = \E [ \oS^{(r)} ]$, $\sigma^2(r) = \Var(\oS^{(r)})$, and 
\eqnst
{ Y^{(r)} 
  = \frac{\oS^{(r)} - \mu(r)}{\sigma(r)}. }
As $r \to \infty$, $Y^{(r)}$ tends to a standard Gaussian. An explicit way to see 
this is to write $\oS^{(r)}$ as a sum of independent contributions from shells:
\eqnst
{ \oS^{(r)}
  = \sum_{n=1}^\infty \left( \oS^{(r+n-1)} - \oS^{(r+n)} \right). }
Since the volume of the $n$-th shell is $\approx c (r+n)^{d-1}$, 
the mean of the $n$-th term is $\approx c (r+n)^{d-1-\gamma}$ and 
its variance is $\approx c' (r+n)^{d-1-2\gamma}$. Asymptotic normality of 
$\oS^{(r)}$ can be deduced for example from Lyapunov's criterion 
\cite[Section X.8, Exercise 3]{FI}. 
We note that the mean and variance of $\oS^{(r)}$ are of the form
$\mu(r) = c \, r^{d-\gamma}$ and $\sigma^2(r) = c' r^{d-2\gamma}$.
Since the value of $\lambda$ can be fixed by scaling, we assume throughout
that $\lambda$ has the fixed value that makes $c' = 1$.

Let $y$ be a fixed possible value of $Y^{(r)}$, and suppose we want to approximate 
the density $f_{Y^{(r)}}(y)$. 
We transform the variable $Y^{(r)}$ into a variable $\tY^{(r)}$, in such a way that the mean 
becomes $\E [ \tY^{(r)} ] = y$, and hence $y$ is `typical' for $\tY^{(r)}$.
This technique is standard in many branches of probability; 
see for example \cite[Section XVI.7]{FII}, 
where it is called the technique of `associated distributions'; or see 
\cite[Section I.3]{dH00}, where it is called the `Cram\'er transform'; other
names are: `change of measure' and `tilting'. In order to define the transformation,
denote the Laplace transform
$\phi_{Y^{(r)}}(t) := \E \left[ e^{-t Y^{(r)}} \right]$. 
The transformed probability density is
\eqn{e:tYr-def}
{ f_{\tY^{(r)}}(z)
  := \frac{e^{\xi z}}{\phi_{Y^{(r)}}(-\xi)} f_{Y^{(r)}}(z), \quad
     z \in \R, }
where the parameter $\xi$ is chosen to be
\eqn{e:Legendre-simple}
{ \xi 
  = \xi(y,r) 
  := \argmin_{\xi' \in \R} \, e^{-\xi' y} \phi_{Y^{(r)}}(-\xi'). } 

Let $\tkappa_2 = \tkappa_2(y,r)$ denote the variance of $\tY^{(r)}$.
It depends both on $r$ and $y$, through the definition of $\xi(y,r)$.
The simplest approximation is to replace $\tY^{(r)}$ 
by a Gaussian of mean $y$ and variance $\tkappa_2$, which gives
$f_{\tY^{(r)}}(y) \approx 1/\sqrt{2 \pi \tkappa_2(y,r)}$.
Inverting the transformation in \eqref{e:tYr-def}, this gives an
approximation of $f_{Y^{(r)}}(y)$.
We find that the relative error of this approximation is
uniform in $y$. That is, 
for all $r \ge 1$ and all $y$ in the support of the distribution of $Y^{(r)}$, 
we have
\eqnsplst
{ f_{Y^{(r)}}(y) 
  = \frac{1}{\sqrt{2 \pi \tkappa_2(y,r)}} e^{- \xi(y,r) \, y} \, \phi_{Y^{(r)}}(-\xi(y,r)) \,
    \left[ 1 + O \left( r^{-d/2} \right) \right], }
where the constant in the error term only depends on $d$. 

When $r$ is large, it is possible to improve the error, by replacing
the normal approximation by the so-called Edgeworth expansion; 
see \cite[Section XVI.2, Theorem 2]{FII}.
We briefly explain the idea of this expansion, in order to state our main 
theorem. The normal approximation 
is based on a Taylor expansion of the characteristic function 
$\chi_{\tY^{(r)}}$ of $\tY^{(r)}$ to second order, and hence involves 
the mean and variance of $\tY^{(r)}$.
When higher order moments also exists, the Taylor expansion can be 
continued with $k$ additional terms for some $k \ge 1$. Abreviating
$\rho = r^{d/2}$, this takes the form 
\eqnspl{e:Edgeworth}
{ \chi_{\tY^{(r)}}(t)
  &= \exp \left( ity - \frac{t^2}{2!} \tkappa_2 + \frac{(it)^3}{3!} \frac{\tkappa_3}{\rho}
    + \dots + \frac{(it)^{k+2}}{(k+2)!} \frac{\tkappa_{k+2}}{\rho^k} 
    + O \left(\frac{t^{k+3}}{\rho^{k+1}}\right) \right) \\
  &= \exp \left( ity - \frac{t^2}{2!} \tkappa_2 \right) \,
    \exp \left( \frac{(it)^3}{3!} \frac{\tkappa_3}{\rho}
    + \dots + \frac{(it)^{k+2}}{(k+2)!} \frac{\tkappa_{k+2}}{\rho^k} 
    + O \left(\frac{t^{k+3}}{\rho^{k+1}}\right) \right) \, }
with some coefficients $\tkappa_n$ that depend on $y$ and $r$. 
Expanding 
the second exponential in the right hand side of \eqref{e:Edgeworth},
according to the exponential power series, and keeping terms of order 
$\rho^{-k}$ and lower, yields the Edgeworth expansion. This is a 
multiplicative polynomial correction to the normal distribution, whose coefficients 
can be expressed in terms of the $\tkappa_n$'s. The asymptotic error of 
the correction is $O( \rho^{-k-1} )$, as $r \to \infty$.
Since we are interested in bounds for finite $r$, we
determine the behaviour of the constant implicit in the $O$ for our
specific case. All constants in our statements will be positive and finite. 
They could be replaced by explicit values throughout, however, we suppress these for the
sake of readability. Constants that we do not need to refer to later on will
be simply denoted $c$ or $C$, and such constants may change on each
appearance. Let $d_1 = 2 - \frac{d}{\gamma} \in (1,2)$.

\begin{theorem}
\label{thm:main}
There exist constants $C_2, C_3$ and $C_4$, and for any $k \ge 0$, there is an 
explicit expression $\tnf_k$, expressible in terms of 
$\tkappa_2, \dots, \tkappa_{k+2}$, such that for $r \ge 1$ and 
$-\frac{\mu(r)}{\sigma(r)} < y < \infty$ we have
\eqnspl{e:ew-density-formula-err}
{ f_{Y^{(r)}}(y) 
  = \tnf_k \, e^{ - \xi y } \, \phi_{Y^{(r)}}(-\xi) \,
    \left[ 1 + \eps_k(r,y) \right], }
where 
\eqnsplst
{ |\eps_k(r,y)| 
  &\le \frac{C_2 \, C_3^k \, k^{k/2}}{(\sqrt{\tkappa_2} \, r^{d/2})^{k+1}} \, 
  \le \frac{C_2 \, C_3^k \, k^{k/2}}{(r^{d/2})^{k+1}}, \quad 
      \text{when $y \ge 0$ and $r^{d/2} \ge C_4 \, \sqrt{k}$;} \\
  |\eps_k(r,y)| 
  &\le C_2 \, C_3^k \, k^{k/2} \, 
      \left( \frac{\tkappa_2^{\frac{1}{d_1} - \frac{1}{2}}}{r^{d/2}} \right)^{k+1} 
  \le \frac{C_2 \, C_3^k \, k^{k/2}}{(r^{d/2})^{k+1}}, \quad \\
  &\qquad\qquad \text{when $y \le 0$ and $r^{d/2} \ge \max \{ C_4 \sqrt{k},\, k \}$.} }
\end{theorem}

\begin{remark}
\label{rem:improve}
An interesting feature of the error bound is that it \emph{improves} away from the mean.
Indeed, $\tkappa_2 > 1$ for $y > 0$ and $\tkappa_2 < 1$ for $y < 0$.
\end{remark}

In Remark \ref{rem:numerical} of Section \ref{sec:results} we explain how 
$\xi(y,r)$ and the required coefficients $\tkappa_n(y,r)$ can be computed efficiently.

\begin{remark}
Note the stronger requirement $r^{d/2} \ge k$ for the estimate on the lower tail
in Theorem \ref{thm:main}. In particular, our methods to prove 
Lemma \ref{lem:upto-order-k-ew-lower} in Section \ref{sssec:lower-tail-ew} 
break down when considering them only under $r^{d/2} \ge C_4 \, \sqrt{k}$. 
We do not know what is the optimal condition on $r$ for the same form of 
bound to hold.
\end{remark}

It may seem restrictive that we are considering here only $\oS^{(r)}$,
in that this is a very specific family of infinitely divisible distributions.
We believe that our methods could be applied in greater generality; a possible
extension would be to replace $r^{-\gamma}$ by $r^{-\gamma} L(r)$ for a
function $L$ slowly varying at infinity.
However, the technicalities in obtaining the error estimates 
are already considerable in our specific case, and hence we do not pursue
more general distributions here.

\subsection{The conditional density of $R_1$ given $S = s$}

In this section, we write down an easily implementable approximation 
for the conditional density 
of $R_1$ given $S = s$, 
and illustrate numerically that it performs well over a large range of values of $s$. 
Let us explain the ideas,
restricting to $d = 2$ and $\gamma = 4$ (when the density of $S$ is known).
Due to scaling properties, we can assume the Poisson density $\lambda$ to
be fixed, and we make the choice $\lambda = 3/\pi$ to make some formulas 
work out nicely.

Let $f_S(s)$ and $f_{R_1}(r_1)$ denote the probability 
densities of $S$ and $R_1$, respectively, and let 
$f_{R_1,S}(r_1,s)$ denote the joint density. 
We are interested in the conditional density
\eqn{e:cond-S-R1}
{ f_{R_1 \,|\, S} (r_1 \,|\, s)
  = \frac{f_{R_1}(r_1) \, f_{\oS^{(r_1)}} (s - r_1^{-4})}{f_S(s)}, \quad
    s^{-1/4} < r_1 < \infty,\, 0 < s < \infty, }
where the densities of $R_1$ and $S$, respectively, are 
(see Section \ref{sec:results} for further details): 
\eqnst
{ f_{R_1}(r_1)
  = 6 \, r_1 \, e^{ - 3 \, r_1^2}, \quad
    r_1 > 0, \qquad \text{ and } \qquad
  f_S(s)
  = \frac{3}{2} \, s^{-3/2} \, 
    \exp \left( - \frac{9 \, \pi}{4 s} \right),
    \quad s > 0. }
%
The mean and variance of $\oS^{(r)}$ are (see Section \ref{sec:results}):
\eqnst
{ \mu(r)
  = 3 \, r^{-2} \qquad \text{ and } \qquad
  \sigma^2(r)
  = r^{-6}. }
Since for large $r$, $\oS^{(r)}$ is approximately Gaussian,  
as a crude approximation, we can try to replace $\oS^{(r_1)}$ 
in formula \eqref{e:cond-S-R1} by a Gaussian. 
However, this does not give a satisfactory result numerically. 
We can improve the approximation with the following three ideas.
\begin{itemize}
\item[(i)] We integrate over the contribution of the nearest few 
points, say the nearest four. Conditionally on 
$R_i = r_i$, $i = 1, \dots, 4$, we write 
$S = r_1^{-4} + \dots + r_4^{-4} + \oS^{(r_4)}$, and replace
$f_{\oS^{(r_4)}}$ by a Gaussian density:
\eqnspl{e:normal-approx-1}
{ f_{R_1 \,|\, S} (r_1 \,|\, s)
  &= \frac{1}{f_S(s)} \iiint_{(r_2,r_3,r_4) : r_1 < \dots < r_4}
    f_{R_1,\dots,R_4}(r_1,\dots,r_4) \\ 
  &\qquad\quad \times f_{\oS^{(r_4)}} \left( s - r_1^{-4} - \dots - r_4^{-4} \right)
    \, dr_2 \, dr_3 \, dr_4 \\
  &\approx \frac{1}{f_S(s)} \iiint_{(r_2,r_3,r_4) : r_1 < \dots < r_4}
    f_{R_1,\dots,R_4}(r_1,\dots,r_4) \\
  &\qquad\quad \times 
    \frac{1}{\sigma(r_4)} \, \nf \left( \frac{s - r_1^{-4} - \dots - r_4^{-4} 
    - \mu(r_4)}{\sigma(r_4)} \right) \, dr_2 \, dr_3 \, dr_4, }
where $\nf(x)$ is the standard normal density.
One of the three integrations (an incomplete beta integral) can be 
carried out analytically, when $d/\gamma = 1/2$. The next integration 
presents an elliptic integral. We give these calculations in 
Section \ref{sec:integration}.
In Figure \ref{fig:R1_cond_S_plots} we show the result of carrying out the 
integration over $r_2, r_3, r_4$ in formula \eqref{e:normal-approx-1}, 
compared to a simulation from the conditional distribution.
\begin{figure}
	\centering
		\includegraphics[scale=0.5]{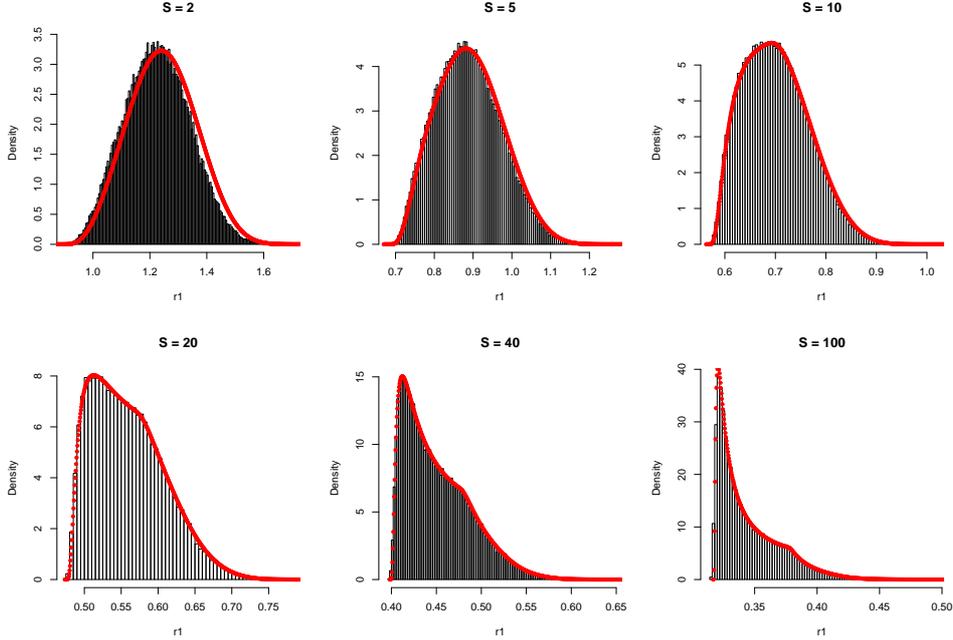}
\caption{The normal approximation of \eqref{e:normal-approx-1}
compared to a direct simulation from the conditional distribution.}		
	\label{fig:R1_cond_S_plots}
\end{figure}
The approximation compares well with the simulation over 
different values of $s$.

\item[(ii)] Via the transformation in \eqref{e:tYr-def}, 
the point of approximation becomes the mean. This allows us to get 
uniform relative error in the approximation, at the expense of 
having to compute $\xi(y,r)$. 

\item[(iii)] The precision of the approximation can be increased
arbitrarily, in principle, if we integrate numerically over $r_i$, 
$i = 1, \dots, \ell$, and approximate $f_{\oS^{(r_\ell)}}$.
The order of the error improves, if we replace the normal
approximation by the Edgeworth expansion of an appropriate order 
depending on $\ell$.
\end{itemize} 

The next theorem formalizes the above approximation scheme and gives a
rigorous error estimate.
Fix a number $a_0$ such that $0 < a_0 < (d/\gamma)/d_1$. Let 
\eqnsplst
{ \ff^{\ell,k}_{R_1,S} (r_1,\, s)
  &:= \int_{r_1}^\infty dr_2 \, \cdots 
    \int_{r_{\ell-1}}^\infty dr_{\ell} \, 
    f_{R_1, \dots, R_{\ell}} \left( r_1, \dots, r_{\ell} \right) \,
    \fg_{\ell,k}(y, r_\ell), }
where
\eqnsplst
{ \fg_{\ell,k}(y,r)
  &= \begin{cases}
     \frac{1}{\sigma(r)} \,    
     \tnf_k \,  
     e^{ - \xi(y,r) y } \, \phi_{Y^{(r)}}(-\xi(y,r))  
     & \text{when $r > (a_0 \, \ell)^{1/d}$;} \\
     0 & \text{when $r \le (a_0 \, \ell)^{1/d}$} 
     \end{cases} }
and
\eqnst
{ y 
  = y(s, r_1, \dots, r_\ell)
  = \frac{s - r_1^{-\gamma} - \dots - r_\ell^{-\gamma} - \mu(r_\ell)}{\sigma(r_\ell)}. }

\begin{theorem}
\label{thm:combined}
There are constants $C_5, C_6, c_1$ such that when $\ell \ge C_5$ and
$k = \lfloor \sqrt{a_0 \, \ell} \rfloor$, then we have
\eqn{e:diff-DF}
{ \left| \prob [ R_1 \le r \,|\, S = s ] 
    - \frac{1}{f_S(s)} \, \int_{s^{-1/\gamma}}^r \ff^{\ell,k}_{R_1,S} (r_1,\, s) \, dr_1 \right|
  \le C_6 \, e^{-c_1 \sqrt{\ell} \log \ell}, \quad r > s^{-1/\gamma},\, s > 0. }
\end{theorem}

\begin{remark}
The above theorem applies just as well, if the path-loss function behaves
differently in a neighbourhood of $0$, for example, if 
\eqnst
{ f(r)
  = \begin{cases}
    r^{-\gamma} & \text{when $r > r_0$;} \\
    r_0^{-\gamma} & \text{when $0 \le r \le r_0$,}
    \end{cases} }
and $S = \sum_k f(R_k)$. For this case, we merely have to 
choose $C_5$ large relative to $r_0$, and change the 
definition of $y$ to 
$(s - f(r_1) - \dots - f(r_\ell) - \mu(r_\ell))/\sigma(r_\ell)$. 
In this case, no series expansion would be available to compute 
$f_S(s)$. Instead, the approximation 
$f_S(s) \approx \int_0^\infty \ff^{\ell,k}_{R_1,S} (r_1,\, s) \, dr_1$
can be used.
\end{remark}

The following stochastic comparison plays a key role in 
ensuring that the error bound in Theorem \ref{thm:combined} is
uniform in $s$. Let $(R'_1, R'_2, \dots)$ have the
same law as $(R_1, R_2, \dots)$.

\begin{theorem}
\label{thm:stoch-mon}
Let $\gamma > d$ and let $0 < s < \infty$. 
There is a coupling between the conditional law of
the collection $(R_1, R_2, \dots)$ given $S = s$ 
and the unconditional law of the
collection $(R'_1, R'_2, \dots)$, such that a.s.~we have 
$R'_{i-1} \le R_i$ for all $i \ge 2$.
\end{theorem}


\subsection{Related works} 

A lower bound on $R_1$ is $S^{-1/\gamma}$, and
when $S$ is large, this is a good first approximation.
The first term in the right hand side of \eqref{e:signal} dominates the sum,
in the sense that $R_1^{-\gamma}$ has the same tail behaviour as $S$
\cite[Section 1.4]{STbook}.
The error of the simple heuristic $R_1 \approx S^{-1/\gamma}$ for
large $S$ is considered in \cite{SW15,MBDJ18}.

A more general setup than considered above is to study a random field of the form 
$S_f(x) = \sum_{P} f(P - x)$, $x \in \R^d$, where the summation 
is over points $P$ of the Poisson process (called Poisson shot noise). 
If $f$ is the function $f_r(y) = |y|^{-\gamma} \, \mathbf{1}_{|y| > r}$,
then $\oS^{(r)} = S_{f_r}(o)$, where $o$ is the origin.
Rice \cite[Section 1.6]{R44} proved that under certain general conditions on $f$, 
$S_{f}(o)$ approaches a normal law as the Poisson density approaches infinity.
It follows from this that $S_{f_r}(o)$ is asymptotically normal as $r \to \infty$
(after rescaling the Poisson process so that $r$ becomes $1$).
Rice also states the Edgeworth expansion around this normal limit. 
Lewis \cite{L73} gives error estimates (for a slightly modified version)
for general $f$ and all orders, with the dependence on $f$ implicit.
Explicit error estimates for the normal and Edgeworth approximation 
of infinitely divisible distributions, inclusing Poisson shot noise,
were considered by Lorz and Heinrich \cite{LH91}. They considered the
supremal additive error in approximating the distribution function, with 
error estimates given in detail for the second order Edgeworth approximation.  
A novelty of our work is that we provide details of the estimate
for all orders, giving the dependence of the error term on the order.
By considering the transformed distribution 
we get very good relative error estimates (the relative error 
improves away from the mean).
All constants in our estimates could be made explicit (with tedious but
straightforward arguments), but we refrained from doing so for the sake
of readability.

A possible alternative approach to approximating $\oS^{(r)}$ 
would be to find a suitable series expansion for the density. 
Feller \cite[Section XVII.6]{FII} gives a pointwise convergent series
for stable densities, and Zolotarev \cite{Zbook} studies their analytical 
properties in detail. We note that, unlike in the case of the stable 
random variable $S$, the logarithm of the characteristic function of $\oS^{(r)}$ 
is no longer given in closed form,
which we believe makes it more problematic to derive a useful series.
Another alternative would be to generalize 
the series of Brockwell and Brown \cite{BB78} using Laguerre polynomials.
We believe that the merit of our approach compared to these possibilities is that 
it is more probabilistic, and that the Edgeworth estimates we develop here are
of interest in their own right, and possibly apply to more general
infinitely divisible families.

In the context of wireless applications, Baccelli and Biswas \cite{BB15} consider 
the joint distribution of the signals measured at a finite number of points 
(with a more general pathloss function than ours), and show asymptotic 
independence as the Poisson density approaches infinity. 
They also consider percolative properties of a random graph defined in terms of 
signal-to-interference ratios.

\bigbreak

The rest of the paper is organized as follows.
In Section \ref{sec:results} we collect some preliminary results and define 
the quantities appearing in Theorem \ref{thm:main}. 
In Section \ref{sec:both-tails} we prove Theorem \ref{thm:main}.
In Section \ref{sec:stoch-mon}
we prove Theorem \ref{thm:stoch-mon} and use it to prove 
Theorem \ref{thm:combined} building on the technical estimate
of Theorem \ref{thm:main}. 
In Section \ref{sec:integration} we show that when $d/\gamma = 1/2$ and 
we take $\ell \ge 4$, one integration over $(r_1, \dots, r_\ell)$ can
be carried out analytically.

\section{Preliminaries}
\label{sec:results}

Recall that $R_1 < R_2 < \dots$ are the radii of the points of 
a Poisson process in $\R^d$ of intensity $\lambda > 0$. It will sometimes be
convenient to consider the following finite version: let 
$R^{(n)}_1 < \dots < R^{(n)}_n$ be the radii of $n$ independent
points chosen uniformly at random from the ball of volume $n / \lambda$
centred at the origin.

Recall that $\oS^{(r)} = \sum_{i : R_i > r} R_i^{-\gamma}$, $r \ge 0$. 
The following lemma, whose proof follows easily from large deviation bounds for 
Binomial and Poisson variables, and is left to the reader,
implies in particular that the sum 
defining $\oS^{(r)}$ converges almost surely for all $\gamma > d$, $r \ge 0$.

\begin{lemma}
\label{lem:large-dev-bnd}
There exist constants $c = c(\lambda,d)$ and $C = C(\lambda,d)$ such that for all 
$k \ge 0$ and $r \ge 0$, we have
\eqnspl{e:large-dev}
{ \prob \left[ \# \{ j : r + k \le R^{(n)}_j < r + k + 1 \} \ge C (r + k)^{d-1} \right]
  &\le e^{- c (r + k)} \\
  \prob \left[ \# \{ j : r + k \le R_j < r + k + 1 \} \ge C (r + k)^{d-1} \right]
  &\le e^{- c (r + k)}. }
\end{lemma}

We will need to approximate $S$ by
$S^{(n)} := \sum_{j=1}^n \left( R^{(n)}_j \right)^{-\gamma}$.
The following lemma provides a quantitative estimate on the rate of convergence. 
An estimate of this type was proved by Cram\'er \cite{C62} (who gave the
details in the symmetric stable case). For the sake of
being self-contained, we provide a proof in Appendix \ref{a:calc}. 

\begin{lemma}
\label{lem:density-conv}
There exist $C = C(\gamma,d)$, $\delta = \delta(\gamma,d)$ such that 
$\sup_{s \in \R} \left| f_{S^{(n)}}(s) - f_{S}(s) \right| \le C \, n^{-\delta}$.
\end{lemma}

Recall that
\eqnsplst
{ Y^{(r)} 
  &= \frac{ \oS^{(r)} - \mu(r) }{\sigma(r)}, \qquad \text{where }
  \mu(r) = \E \left[ \oS^{(r)} \right], \quad \sigma^2(r) = \Var \left(\oS^{(r)}\right). }
Note that since $\oS^{(r)} > 0$ a.s., we have $Y^{(r)} > -\frac{\mu(r)}{\sigma(r)}$ a.s. 

Next, we compute the Laplace transforms of $\oS^{(r)}$ and $Y^{(r)}$.
Writing $|\bx|$ for the Euclidean norm of $\bx \in \R^d$, 
and $\omega_{d-1}$ for the measure of the $d-1$-sphere, 
we have:
\eqnsplst
{ \varphi_{\oS^{(r)}}(t)
  &:= \E \left[ e^{-t \oS^{(r)}} \right]
  = \exp \left( \int_{\{ \bx : |\bx| > r \}} \left( 
    e^{-t \, |\bx|^{-\gamma}} - 1 \right) \lambda\, d\bx \right) \\
  &= \exp \left( \lambda \, \omega_{d-1} 
    \int_{r}^\infty q^{d-1} \, \left( e^{-t q^{-\gamma}} - 1 \right) \, 
    dq \right). }
When $r = 0$, this shows that 
$\varphi_{S}(t) = \exp ( - b_1 \, t^{d/\gamma} )$, 
with $b_1(\lambda,\gamma,d) = \frac{\lambda \, \omega_{d-1}}{d} \Gamma(1 - d/\gamma) > 0$, 
and hence $S$ has a one-sided stable distribution of index $d/\gamma$
\cite[Section XVII.5]{FII}.
For $r > 0$, we change variables via $q = r \, u^{-1/\gamma}$, 
which gives
\eqnsplst
{ \varphi_{\oS^{(r)}}(t)
  &= \exp \left( \frac{\lambda \, \omega_{d-1}}{\gamma} \, 
     r^d \, \int_0^1 \left( e^{- r^{-\gamma} t u} - 1 \right)\, 
     u^{-1-d/\gamma} \, du \right). }
Letting $a_1 = a_1 (r) = (\lambda \omega_{d-1} / \gamma) \, r^d$, and 
$a_2 = a_2 (r) = r^{-\gamma}$, we can then write 
\eqn{e:Laplace-scaled}
{ \varphi_{\oS^{(r)}}(t)
  = \exp \left( a_1 \, \psi ( a_2 t ) \right), }
where 
\eqnspl{e:psi-formula}
{ \psi(a_2 \, t)
  &= \int_0^1 \left( e^{-a_2 t u} - 1 \right) \, u^{-1-d/\gamma} \, du 
  = \sum_{n=1}^\infty \frac{(-a_2)^n t^n}{n! \, (n - d/\gamma)}. }
In particular, the mean and variance of $\oS^{(r)}$ are
\eqnsplst
{ \kappa_1
  = \kappa_1(r)
  = \mu(r)
  = \frac{a_1 a_2}{1 - d/\gamma} \qquad \text{ and } \qquad  
  \kappa_2
  = \kappa_2(r)
  = \sigma^2(r)
  = \frac{a_1 a_2^2}{2 - d/\gamma}, }
and the higher order cumulants are given by
\eqnst
{ \kappa_n
  = \kappa_n(r)
  = \frac{a_1 a_2^n}{n-d/\gamma}, \quad n \ge 3. }

Since the value of $\lambda$ can be fixed by scaling, we specialize to 
$\lambda = (2\gamma - d)/\omega_{d-1}$, which 
yields the simple form $\sigma(r) = r^{\frac{d}{2} - \gamma}$. Letting 
$d_1 = d_1(d,\gamma) := 2 - \frac{d}{\gamma}$, and recalling the notation 
$\rho = r^{d/2}$, we can write $\phi_{Y^{(r)}}(t)$ in the form:
\eqnspl{e:Laplace-scaled-2}
{ \phi_{Y^{(r)}}(t)
  &= \exp \left( d_1 \, r^d \, \psi_0 
    \left( \frac{t}{r^{d/2}} \right) \right) 
  = \exp \left( d_1 \, \rho^2 \, \psi_0 
    \left( \frac{t}{\rho} \right) \right) \qquad \text{where} \\
  \psi_0(s)
  &:= \int_0^1 (e^{-su}-1-su) \, u^{-1-d/\gamma} \, du
  = \sum_{n=2}^\infty \frac{(-s)^n}{n! \, (n - \frac{d}{\gamma})}. }
We now collect some estimates for the characteristic function of $Y^{(r)}$.
From \eqref{e:Laplace-scaled-2} we have that
the Fourier transform of $Y^{(r)}$ is given by
\eqnspl{e:chi-Yr}
{ \chi_{Y^{(r)}}(t)
  &= \exp \left( - \frac{t^2}{2} + \sum_{n=3}^\infty 
    \kappa'_n(r) \frac{(it)^n}{n!} \right)
  = \exp \left( d_1 \, r^d \, \sum_{n=2}^\infty 
    \frac{(it/r^{d/2})^n}{(n - d/\gamma) \, n!} \right) \\
  &= \exp \left( d_1 \, r^d \, \int_0^1 \left( 
    e^{itu/r^{d/2}} - 1 - \frac{itu}{r^{d/2}} \right) \, u^{-1 - d/\gamma} \, du \right) 
  = \left( \chi_{Y^{(1)}}(t/r^{d/2}) \right)^{r^d}, }
where $\kappa'_n(r)$ are the cumulants of $Y^{(r)}$.

We will need estimates on $\chi_{Y^{(r)}}$ away from the real axis.

\begin{lemma}
\label{lem:chi-complex1} \ \\
(i) For any $\xi > 0$ we have
\eqnst
{ \left| \chi_{Y^{(r)}}(t - i \xi) \right|
  \le \phi_{Y^{(r)}}(-\xi) \, \left| \chi_{Y^{(r)}}(t) \right|. }
(ii) For any $\zeta > 0$ we have
\eqnst
{ \left| \chi_{Y^{(r)}}(t + i \zeta) \right|
  \le \phi_{Y^{(r)}}(\zeta) \, 
     \left| \chi_{Y^{(r)}}(t) \right|^{e^{-\zeta/r}}. }
(iii) For any $\xi, \zeta > 0$ we have
\eqnst
{ \lim_{|t| \to \infty} \, \sup_{-\zeta \le \mu \le \xi} \,
    \left| \chi_{Y^{(r)}}(t - i \mu) \right|
  = 0. }
\end{lemma}

\begin{proof} 
For any $\xi \in \R$ we have
\eqnspl{e:bound-Re-log}
{ &\Re \log \chi_{Y^{(r)}}(t - i \xi) \\
  &\quad = d_1 \, r^d \, \Re \int_0^1
    \left( e^{(itu + \xi u)/r^{d/2}} - 1 - \frac{itu + \xi u}{r^{d/2}} \right)
    u^{-1-d/\gamma} \, du \\
  &\quad = d_1 \, r^d \, \int_0^1
    e^{\xi u/r^{d/2}} \left( \cos \left(\frac{tu}{r^{d/2}}\right) - 1 \right) 
    u^{-1-d/\gamma} \, du 
    + \int_0^1 \left( e^{\xi u / r^{d/2}} 
    - \frac{\xi u}{r^{d/2}} - 1 \right) u^{1-d/\gamma} \, du \\
  &\quad \le \min \left\{ e^{\xi/r^{d/2}}, 1 \right\} \, 
    d_1 \, r^d \, \int_0^1
    \left( \cos \left(\frac{tu}{r^{d/2}}\right) - 1 \right) \, u^{-1-d/\gamma} \, du 
    + \log \phi_{Y^{(r)}}(-\xi) \\
  &\quad = \min \left\{ e^{\xi/r^{d/2}}, 1 \right\} \, 
    \Re \log \chi_{Y^{(r)}}(t) + \log \phi_{Y^{(r)}}(-\xi). }
When $\xi > 0$, we obtain (i) immediately, and statement (ii) follows 
by taking $\xi = - \zeta$.
Statement (iii) follows from (i), (ii) and the 
facts that $\lim_{|t| \to \infty} | \chi_{Y^{(r)}}(t) | = 0$,
and $\sup_{-\zeta \le \mu \le \xi} \phi_{Y^{(r)}}(-\mu) < \infty$.
\end{proof}

In order to approximate $f_{Y^{(r)}}(y)$ for a given $y$, we 
consider $\tY^{(r)}$ introduced in \eqref{e:tYr-def}.
Recall the expression: 
\eqnsplst
{ 
  f_{\tY^{(r)}}(z)
  &= \frac{e^{\xi(y,r) z}}{\phi_{Y^{(r)}}(-\xi(y,r))} f_{Y^{(r)}}(z), \quad
     z \in \R. }
The parameter $\xi = \xi(y,r)$ is the solution to the equation
\eqn{e:sol-y-2}
{ \frac{y}{\rho} 
  = - d_1 \, \psi_0'(-\xi/\rho) 
  = d_1 \, \sum_{n=2}^\infty 
    \frac{( \xi / \rho )^{n-1}}{(n-1)! \, (n - \frac{d}{\gamma})}. }
From the definition of $\tY^{(r)}$, we have that the mean and 
higher order cumulants of $\tY^{(r)}$ are 
$y$ and $\tkappa_n/\rho^{n-2}$, where
\eqnspl{e:tkappa_k}
{ \E [ \tY^{(r)} ] = y \qquad \text{ and } \qquad
  \tkappa_n
  := d_1 \, \int_0^1 e^{\xi u / r^{d/2}} \, u^{n - 1 - d/\gamma} \, du, \quad n \ge 2, }

We have the following relationship between $\chi_{Y^{(r)}}$ and $\chi_{\tY^{(r)}}$:
\eqn{e:chis}
{ \chi_{Y^{(r)}}(t-i\xi)
  = \phi_{Y^{(r)}}(-\xi) \int_{-\infty}^\infty e^{itz} 
    \frac{e^{\xi z}}{\phi_{Y^{(r)}}(-\xi)} f_{Y^{(r)}}(z) \, dz
  = \phi_{Y^{(r)}}(-\xi) \chi_{\tY^{(r)}}(t). }
From this we obtain
\eqnspl{e:change-path}
{ f_{Y^{(r)}}(y)
  &= \frac{1}{2 \pi} \int_{-\infty}^\infty e^{-ity} \, \chi_{Y^{(r)}}(t) \, dt
  = \frac{1}{2 \pi} \int_{-\infty}^\infty e^{-ity - \xi y} \, 
    \chi_{Y^{(r)}}(t-i\xi) \, dt \\
  &= \frac{e^{-\xi y} \phi_{Y^{(r)}}(-\xi)}{2 \pi} 
    \int_{-\infty}^\infty e^{-ity} \, 
    \chi_{\tY^{(r)}}(t) \, dt. }
In the second step we moved the path of integration, which is justified by
Lemma \ref{lem:chi-complex1}(iii). 
Our goal will be to estimate the expression
\eqnspl{e:tchi-expression-1}
{ e^{-ity} \, \chi_{\tY^{(r)}}(t)
  = \exp \left( d_1 \, \rho^2 \, \int_0^1 e^{\xi u / \rho}
    \left( e^{itu/\rho} - 1 - \frac{itu}{\rho} \right) \, 
    u^{-1-d/\gamma} \, du \right). }
Recall $\tkappa_k$ defined in \eqref{e:tkappa_k}:
\eqn{e:tkappa-formula}
{ \tkappa_k
  = d_1 \, \int_0^1 e^{\xi u / \rho} \, u^{k-1-d/\gamma} \, du, 
    \quad k \ge 2. }
It is clear from this formula that with $d_3 = d_1 / (1 - \frac{d}{\gamma})$,
for all $-d_3 \, \rho < y < \infty$ we have 
$\tkappa_2 \ge \tkappa_3 \ge \dots > 0$.

\begin{remark}
\label{rem:numerical}
Let us comment on how the approximation in Theorem \ref{thm:main}
can be computed numerically. The functional relationship between 
$y/\rho$ and $\xi/\rho$ does not depend on $r$, and is found to be
\eqnst
{ \frac{y}{\rho}
  = d_1 \, \int_0^1 \left( e^{\xi u / \rho} - 1 \right) \, u^{-d/\gamma}. }
The right hand side can be written in terms of the regularized 
incomplete gamma function, for which efficient numerical evaluation is available
both for positive arguments (when $\xi < 0$, equivalently, $y < 0$) 
\cite{G79}; and for negative argumments (when $\xi > 0$, equivalently, $y > 0$)
\cite{GR-AST16}.
Thus the increasing convex function $(\xi/\rho) \mapsto (y/\rho)$ can be 
easily inverted using Newton's method. The expression $e^{- \xi y} \, \phi_{Y^{(r)}}(-\xi)$
can be written in the form $\exp ( \rho^2 \, h(\xi/\rho, y/\rho) )$
with an explicit function $h$. The number $\tkappa_2$ as well as all $\tkappa_n$'s, 
are again given in terms of incomplete gamma functions.
\end{remark}

As mentioned earlier, for large $r$, the order of the approximation can be improved, if we
replace the normal approximation by an Edgeworth expansion \cite[Section XVI]{FII},
and we are now ready to define the quantities appearing in Theorem \ref{thm:main}.
A reader not familiar with the expansion should note that it is possible to
follow our proof of Theorem \ref{thm:main} without prior exposure, as it is
self-contained. On the other hand, the somewhat complicated expressions
one needs to define (see \eqref{e:tcjl-def} and \eqref{e:tnfk-def} below) may become
more transparent upon reading \cite[Section XVI]{FII}.

Recall that $\nf$ denotes the standard normal density, and for $k = 0$ define
\eqnst
{ \tnf_0
  := \frac{1}{\sqrt{\tkappa_2}} \, \nf(0). }
For $k \ge 1$, and $j \ge 1$, $1 \le \ell \le j$ define
\eqn{e:tcjl-def}
{ \tc_{j,\ell}
  = \frac{1}{\ell!} \!\!\!\!\! 
     \sum_{\substack{3 \le n_1, \dots, n_\ell \le k+2 \\
     (n_1-2) + \dots + (n_\ell-2) = j}} \frac{\tkappa_{n_1}}{n_1!} 
     \cdots \frac{\tkappa_{n_\ell}}{n_\ell!}. }
Let 
\eqnspl{e:tnfk-def}
{ \tnf_k
  := \frac{1}{\sqrt{\tkappa_2}} \nf(0) 
     \left[ 1 + \sum_{j = 1}^k \frac{(-1)^j}{(r^{d/2})^j} \, \sum_{\ell=1}^j \, 
     \frac{\tc_{j,\ell}}{\tkappa_2^{\frac{j}{2} + \ell}} \, 
     H_{j+2\ell}(0) \right], \quad k \ge 1,}
where $H_n(z)$, $n \ge 0$, denotes the $n$-th Hermite polynomial, that is,
$(-1)^n H_n(z) \nf(z)$ has Fourier transform $(it)^n \, e^{-t^2/2}$.
Observe that in computing $\tnf_k(0)$, all terms with odd $j$ vanish.
In particular, $\tnf_{2k+1} = \tnf_{2k}$ for $k = 0, 1, \dots$.

We close this section with collecting a few estimates that we will often use.
We have \cite[Eqn.~XVI.(2.8)]{FII}:
\eqn{e:exp-alpha-beta}
{ \left| e^{\alpha} - e^{\beta} \right|
  \le e^{-x} \left| \alpha - \beta \right| \, 
      \exp \left( \max \{ |\alpha+x|,\, |\beta+x| \} \right), \quad
      \alpha, \beta \in \C,\, x \in \R. }
The following estimate is standard for the tail of the 
normal distribution (see for example \cite[Theorem (1.4)]{Durrett}):
\eqn{e:normal-tail-est}
{ \int_x^\infty \exp ( -t^2 a/2 ) \, dt
  \le \frac{1}{a \, x} \, \exp ( - x^2 a /2 ). }
We will use the following Gaussian absolute moments:
\eqn{e:Gauss-moments}
{ \int_{-\infty}^\infty |t|^k \, \exp ( - t^2 a / 2) \, dt
  = \left(\frac{2}{a}\right)^{\frac{k+1}{2}} \, 
    \Gamma \left( \frac{k+1}{2} \right). }

\section{Estimates of the transformed distribution}
\label{sec:both-tails}

In this section we prove Theorem \ref{thm:main}. 
The arguments for $k = 0$ (normal approximation) are similar and simpler than 
for $k \ge 1$. Therefore, we only give the details when $k \ge 1$.

%

We will often use Stirling's formula \cite[Section II.9]{FI} in the form:
\eqn{e:Stirling}
{ \sqrt{2 \pi N} N^N e^{-N}
  \le \Gamma(N + 1)
  \le \sqrt{2 \pi N} N^N e^{-N + \frac{1}{12 N}}, \qquad N \ge 1. }

\subsection{Upper tail}
\label{sssec:upper-tail-ew}

In this section we prove Theorem \ref{thm:main} in the case $y \ge 0$.
It follows then from the definition of $\xi = \xi(y,r)$ that 
$0 \le \xi(y,r) < \infty$, and from \eqref{e:tkappa-formula}
that $\tkappa_2 \ge 1$. Recall that $\rho = r^{d/2}$. 
Recall the formula \eqref{e:change-path} 
and that our goal is to approximate
\eqnspl{e:tchi-expression-ew}
{ e^{-ity} \, \chi_{\tY^{(r)}}(t)
  = \exp \left( d_1 \, \rho^2 \, \int_0^1 e^{\xi u / \rho}
    \left( e^{itu/\rho} - 1 - \frac{itu}{\rho} \right) \, u^{-1-d/\gamma} \, du \right)
  =: e^{u(t)}. }
We define
\begin{align}
\label{e:qkrk-etc-ew}
  {q}_k(t)
  &= \rho^2 \, \sum_{n=3}^{k+2} \frac{(it/\rho)^n \, \tkappa_n}{n!} & 
  {r}_k(t)
  &= \rho^2 \, \sum_{n=k+3}^{\infty} \frac{(it/\rho)^n \, \tkappa_n}{n!} \\
  {u}(t)
  &= -\frac{t^2}{2} \tkappa_2 + {q}_k(t) + {r}_k(t) & 
  {\widetilde{\chi}}_k(t)
  &:= \exp \left( - \frac{t^2}{2} \tkappa_2 + {q}_k(t) \right). 
\end{align}
We give separate estimates for $|t| \le \rho / 2$ and 
$|t| > \rho / 2$. For $|t| \le \rho / 2$, 
we estimate $e^{-ity} \, \chi_{\tY^{(r)}}(t) = e^{{u}(t)}$ first by 
${\tchi}_k(t) = \exp \left( - \frac{t^2}{2} \, \tkappa_2 + {q}_k(t) \right)$,
then by the Taylor expansion of $e^{{q}_k(t)}$. Keeping terms of order
$\rho^{-k}$ and lower yields the Edgeworth expansion. 
The above steps are carried out in a series of lemmas.

\begin{lemma}
\label{lem:truncate-ut-both}
There exists a constant $C$ such that for all $r \ge 1$ and
$k \ge 1$ we have
\eqnsplst
{ \int_{- \rho / 2}^{\rho / 2}
    \left| e^{{u}(t)} - {\tchi}_k(t) \right| \, dt
  \le \frac{1}{\sqrt{\tkappa_2}} \, 
    \frac{C^k \, k^{-\frac{k}{2}}}{(\sqrt{\tkappa_2} \, \rho)^{k+1}}. }
\end{lemma}

\begin{proof}
Let $\alpha(t) = -\frac{t^2 \, \tkappa_2}{2} + q_k(t)$, $\beta(t) = u(t)$ 
and $x(t) = \frac{t^2 \, \tkappa_2}{2}$.
For all $t \in \R$, $0 \le u \le 1$, $r \ge 1$ and $k \ge 1$ we have
\eqnsplst
{ \left| \left( e^{itu/\rho} - 1 - \frac{itu}{\rho} \right) 
    - \left( \sum_{n=2}^{k+2} \frac{(itu)^n}{\rho^{n} \, n!} \right) \right|
  \le \frac{|t|^{k+3} \, u^{k+3}}{\rho^{k+3} \, (k+3)!}. }
Multiplying by $d_1 \, \rho^2 \, e^{\xi u / \rho} u^{-1-d/\gamma}$, integrating 
over $0 \le u \le 1$, and using Stirling's formula yields that  
for $|t| \le \rho / 2$, we have
\eqnsplst
{ \left| \beta(t) - \alpha(t) \right|
  \le t^2 \, \frac{|t|^{k+1}}{\rho^{k+1} \, (k+3)!} \, \tkappa_{k+3}. }
We also have
\eqnspl{e:max-qk-rk}
{ &\max \{ |\alpha(t) + x(t)|,\, |\beta(t) + x(t)| \}
  = \max \{ |q_k(t)|,\, |q_k(t) + r_k(t)| \}
  \le \rho^2 \, \sum_{n=3}^\infty \frac{|t|^n \, \tkappa_n}{\rho^{n} \, n!} \\
  &\qquad \le \frac{|t|^3}{\rho} \, \tkappa_3 \, \sum_{n=3}^\infty 
     \frac{|t|^{n-3}}{\rho^{n-3} \, (n-3)! \, 6}
  \le \frac{t^2}{2} \, \tkappa_2 \, \frac{|t|}{3 \, \rho} \, e^{|t|/\rho}
  \le \frac{t^2}{4} \, \tkappa_2. }
Using \eqref{e:exp-alpha-beta}  
and then \eqref{e:Gauss-moments} and Stirling's formula \eqref{e:Stirling}
yields:
\eqnsplst
{ \int_{- \rho / 2}^{\rho / 2}
     \left| \chi_{\tY^{(r)}}(t) - {\tchi}_k(t) \right| \, dt 
  &\le \frac{\tkappa_{k+3}}{\rho^{k+1} \, (k+3)!} \,
     \int_{-\infty}^\infty |t|^{k+3} \, e^{-t^2 \tkappa_2 /4} \, dt 
  \le \frac{C^k \, k^{-k/2}}{\rho^{k+1} \, \tkappa_2^{\frac{k}{2} + 1}}. }
\end{proof}

Consider now 
${\tchi}_k(t) = \exp ( - \frac{t^2}{2} \tkappa_2 + {q}_k(t) )$.
Expand $e^{{q}_k(t)}$ using the exponential power series, 
and collect terms according to inverse powers of $\rho$. This yields unique polynomials 
${p}_1, {p}_2, \dots$ with real coefficients, such that 
\eqn{e:tchik-both}
{ {\widetilde{\chi}}_k(t)
  = e^{-t^2 \tkappa_2 /2} \, \left[ 1 + \sum_{j=1}^k \frac{{p}_j(it)}{\rho^{j}} 
    + {\tilde{r}}_k(t) \right]
  = \chi_k(t) + e^{-t^2 \tkappa_2 /2} {\tilde{r}}_k(t), }
where ${\tilde{r}}_k$ collects all the terms of order 
$\rho^{-(k+1)}$ and higher, and $\chi_k(t)$ is defined by the second equality.
More explicitly, recalling \eqref{e:tcjl-def}, we have
\begin{align*}
  {p}_j(it)
  &= \sum_{\ell=0}^j \tc_{j,\ell} \, (it)^{j+2\ell}.
\end{align*}
Recall that the main term for the approximation is
$\tilde{\nf}_k = \frac{1}{2 \pi} \int_{-\infty}^\infty \chi_k(t) \, dt$.
The following simple upper bound will be useful.

\begin{lemma}
\label{lem:tcjl}
We have
\eqnst
{ \tc_{j,\ell}
  \le \frac{\left( \tkappa_3 \right)^\ell \, \ell^{j}}{\ell! \, j!}. }
\end{lemma}

\begin{proof}
Recalling the definition from \eqref{e:tcjl-def} we have
\eqnspl{e:tcjl-bound}
{ \tc_{j,\ell}
  &= \frac{1}{\ell!}
    \sum_{\substack{3 \le n_1, \dots, n_\ell \le k+2 : \\
    (n_1 - 2) + \dots + (n_\ell - 2) = j}} 
    \frac{\tkappa_{n_1} \, \cdots \, \tkappa_{n_\ell}}{n_1! \cdots n_\ell!} 
  \le \frac{\left( \tkappa_3 \right)^\ell}{\ell! \, j!} 
    \sum_{\substack{1 \le m_1, \dots, m_\ell \le k : \\
    m_1 + \dots + m_\ell = j}} 
    \frac{j!}{m_1! \cdots m_\ell!} 
 \le \frac{\left( \tkappa_3 \right)^\ell \, \ell^{j}}{\ell! \, j!}. }
\end{proof}

We now find an estimate for the error resulting from
omitting ${\tilde{r}}_k(t)$ in \eqref{e:tchik-both}. We do this in 
two steps, summarized in the following two lemmas.

\begin{lemma}
\label{lem:truncate-exp-both}
There exists a constant $C$ such that for all $r \ge 1$ and
$k \ge 1$ we have
\eqn{e:exp-diff-both}
{ \int_{-\rho/2}^{\rho/2} e^{-t^2 \tkappa_2 / 2} \, \left| e^{{q}_k(t)} 
    - \sum_{\ell=0}^k \frac{{q}_k(t)^\ell}{\ell!} \right| \, dt
  \le \frac{C^k \, k^{k/2}}{\sqrt{\tkappa_2} \, (\sqrt{\tkappa_2} \rho)^{k+1}}. }
\end{lemma}

\begin{proof}
Recalling \eqref{e:max-qk-rk} we have
\eqnst
{ \left| {q}_k(t) \right|
  \le \frac{|t|^3 \, e^{|t|/\rho}}{6 \, \rho} \, \tkappa_2 
  \le \frac{t^2}{4} \, \tkappa_2. }
Hence the expression inside absolute values in the left hand 
side of \eqref{e:exp-diff-both} is at most
\eqnspl{e:qkt-bnd-both}
{ \frac{\left| {q}_k(t) \right|^{k+1}}{(k+1)!} \, e^{|{q}_k(t)|}
  &\le \frac{\tkappa_2^{k+1}}{\rho^{k+1}} \, \frac{|t|^{3(k+1)}}{(k+1)!} \, 
      \left( \frac{e^{1/2}}{6} \right)^{k+1} \, e^{t^2 \tkappa_2 / 4} 
  \le \frac{C^k \, \tkappa_2^{k+1}}{\rho^{k+1}} \, \frac{1}{k^k} \, 
      |t|^{3(k+1)} \, e^{t^2 \tkappa_2 / 4}. }
Using \eqref{e:Gauss-moments} and Stirling's formula, we have
\eqnst
{ \int_{-\rho/2}^{\rho/2} |t|^{3(k+1)} e^{-t^2 \tkappa_2 / 4} \, dt
  \le \left( \frac{4}{\tkappa_2} \right)^{\frac{3}{2} k + 2} \, 
      \Gamma \left( \frac{3}{2} k + 2 \right)
  \le C^k \, k^{3k/2} \, \tkappa_2^{-\frac{3}{2}k-2}. }
Taking into account the additional terms in the right hand side
of \eqref{e:qkt-bnd-both}, this implies that the left hand side 
of \eqref{e:exp-diff-both} is bounded by the expression claimed 
in the Lemma.
\end{proof}

\begin{lemma}
\label{lem:upto-order-k-both}
(i) There exists a constant $C$ such that for any 
$k \ge 1$ and $r \ge 1$ we have
\eqn{e:ser-diff}
{ \int_{-\rho/2}^{\rho/2} e^{-\frac{t^2}{2} \tkappa_2} \, 
    \left| 1 + \sum_{j=1}^k \frac{{p}_j(it)}{\rho^{j}} 
    - \sum_{\ell=0}^k \frac{{q}_k(t)^\ell}{\ell!} \right|
  \le \frac{C^k}{\sqrt{\tkappa_2}} \, 
    \sum_{j=k+1}^{k^2} \frac{k^{j/2}}{(\sqrt{\tkappa_2} \rho)^j}. }
(ii) When $\rho \ge \sqrt{k}$, the right hand side of \eqref{e:ser-diff}
is at most $\frac{C^k \, k^{k/2}}{\sqrt{\tkappa_2} (\sqrt{\tkappa_2} \rho)^{k+1}}$.
\end{lemma}

\begin{proof}
(i) We first note that the statement is vacuous when $k = 1$, as the expression
inside absolute values vanishes then. Henceforth we assume $k \ge 2$.
The expression inside absolute values in the left hand side of \eqref{e:ser-diff} 
consists of those terms of $\sum_\ell {q}_k^\ell/\ell!$ where $\rho^{-1}$ occurs with a power 
at least $k+1$. These are 
\eqnspl{e:the-terms}
{ &\sum_{j=k+1}^{k^2} \frac{1}{\rho^j} \sum_{\ell = \lceil j/k \rceil}^{k}
    \widetilde{c}_{j,\ell} (it)^{j+2\ell}. }    
For fixed $j$ and $k$, we take absolute values in \eqref{e:the-terms}, and
integrate. Formula \eqref{e:Gauss-moments} gives
\eqnst
{ \int_{-\rho/2}^{\rho/2} e^{-\frac{t^2}{2} \tkappa_2} \, |t|^{j+2\ell} \, dt
  \le \left( \frac{2}{\tkappa_2} \right)^{\frac{j}{2} + \ell + \frac{1}{2}} \, 
      \Gamma \left(\frac{j}{2} + \ell + \frac{1}{2} \right). }
Using Lemma \ref{lem:tcjl} to bound 
$\widetilde{c}_{j,\ell}$, this yields the following bound on the left
hand side of \eqref{e:ser-diff}:
\eqnspl{e:sum-j-ell}
{ &\frac{1}{\sqrt{\tkappa_2}} \, 
    \sum_{j=k+1}^{k^2} \frac{1}{(\sqrt{\tkappa_2} \rho)^j} \, \sum_{\ell = \lceil j/k \rceil}^{k}
    \frac{\ell^{j} \, 
    2^{\frac{j}{2} + \ell + \frac{1}{2}} \,
    \Gamma(\frac{j}{2} + \ell + \frac{1}{2})}{\ell! \, j!}, }
Fix $k+1 \le j \le k^2$, and consider the ratio of the terms corresponding 
to $\ell+1$ and $\ell$ in the second sum.
This ratio is
\eqnspl{e:ratio}
{ \frac{(\ell+1)^j}{\ell^j} \, \frac{2 \, (\frac{j}{2} + \ell + \frac{1}{2})}{(\ell+1)}
  > 1. }    
Thus the sum over $\ell$ is bounded above by $k$ times the $\ell = k$ term. This gives 
\eqnspl{e:sum-ell-only}
{ &\sum_{\ell = \lceil j/k \rceil}^{k} \frac{\ell^{j} \, 2^{\frac{j}{2} + \ell + \frac{1}{2}} \,
    \Gamma(\frac{j}{2} + \ell + \frac{1}{2})}{\ell! \, j!}
  \le k \, \frac{k^{j} \, 2^{\frac{j}{2} + k + \frac{1}{2}} \,
    \Gamma(\frac{j}{2} + k + \frac{1}{2})}{k! \, j!} \\
  &\qquad \le C^k \, \frac{k^j \, 2^{\frac{j}{2} + k} \, 
    \left( \frac{j}{2} + k \right)^{\frac{j}{2} + k} \,
    e^{-\frac{j}{2}-k}}{k^k \, j^j \, e^{-j}}
  \le C^k \, \frac{k^j \, \left( j + 2k \right)^{\frac{j}{2} + k} \,
    e^{\frac{j}{2}}}{k^k \, j^j}. }
Writing $j = x k$, with $x > 1$, the right hand side of \eqref{e:sum-ell-only} is
\eqnspl{e:sum-ell-only-2}
{ &C^k \, k^{j/2} \, \frac{k^{\frac{x}{2} k} \, \left( (2+x) k \right)^{(\frac{x}{2} + 1) k} \,
    e^{\frac{x}{2} k}}{k^k \, (x k)^{x k}}
  = C^k \, k^{j/2} \, \frac{(2 + x)^{(\frac{x}{2} + 1)k} \, e^{\frac{x}{2} k}}{x^{x k}} \\
  &\qquad = C^k \, k^{j/2} \, \left[ \frac{(2 + x)^{(\frac{x}{2} + 1)} \, 
    e^{\frac{x}{2}}}{x^{x}} \right]^k
  \le C^k \, k^{j/2}. }
This proves statement (i) of the lemma.

(ii) When $\rho \ge \sqrt{k}$, using $\tkappa_2 \ge 1$ 
(which holds due to $y \ge 0$), we have
\eqnsplst
{ \frac{C^k}{\sqrt{\tkappa_2}} \, \sum_{j=k+1}^{k^2} \frac{k^{j/2}}{(\sqrt{\tkappa_2} \, \rho)^j}
  &\le \frac{C^k \, k^{k/2}}{\sqrt{\tkappa_2} \, (\sqrt{\tkappa_2} \, \rho)^{k+1}} \, 
     \sum_{j=k+1}^{k^2} \frac{k^{(j-k)/2}}{\rho^{j-k-1}}
  \le \frac{C^k \, k^{k/2}}{\sqrt{\tkappa_2} \, (\sqrt{\tkappa_2} \, \rho)^{k+1}}. }
\end{proof}

We also need bounds on the tails of $\chi_k(t)$ and $\chi_{\tY^{(r)}}(t)$,
in the range $|t| > \frac{1}{2} \rho$. 

\begin{lemma}
\label{lem:tail-bnd-tYr}
For $|t| > \rho/2$ we have
\eqn{e:tail-bnd-tYr}
{ \left| e^{-ity} \, \chi_{\tY^{(r)}}(t) \right|
  \le \exp \left( - c \, \rho^{d_1} \, |t|^{d/\gamma} \right) \,
      \exp \left( - c \, \rho^2 \, \sqrt{\tkappa_2} \right). }
\end{lemma}

\begin{proof}
We have
\eqnspl{e:tchi-re-log-estimate1}
{ \Re \, \log \left( e^{-ity} \, \chi_{\tY^{(r)}}(t) \right)
  &= d_1 \, \rho^2 \, \int_0^1 e^{\xi u / \rho} \,
    \left( \cos \left( \frac{t u}{\rho} \right) - 1 
    \right) \, u^{-1-d/\gamma} \, du. }
When $0 \le u \le \frac{1}{4} \rho \, |t|^{-1}$, we have 
$|t u / \rho| \le \frac{1}{4}$, and    
$\cos (t u / \rho) \le 1 - t^2 u^2 / 4 \rho^2$. Therefore, the
contribution of this interval of $u$ to the right hand side of
\eqref{e:tchi-re-log-estimate1} is at most
\eqnspl{e:decaying}
{ &- d_1 \, \rho^2 \, \int_0^{\frac{1}{4} \rho |t|^{-1}}
    \frac{t^2 u^2}{4 \rho^2} \, u^{-1-d/\gamma} \, du
  = - \frac{t^2}{4} \, d_1 \, \int_0^{\frac{1}{4} \rho |t|^{-1}}
    u^{1-d/\gamma} \, du \\
  &\qquad = - \frac{t^2}{4} \, d_1 \, \left( \frac{1}{4} \right)^{d_1} \,
    \rho^{d_1} \, |t|^{-d_1} 
  = - c \, \rho^{d_1} \, |t|^{d/\gamma}. }
On the other hand, the contribution of the interval $1/2 \le u \le 1$ 
to the right hand side of \eqref{e:tchi-re-log-estimate1} is at most
\eqnspl{e:oscillating}
{ d_1 \, \rho^2 \, \sqrt{e^{\xi/\rho}} \, 2^{1+d/\gamma} \,
    \int_{1/2}^1 \left( \cos \left( \frac{t u}{\rho} \right) - 1 
    \right) \, du
  \le - c \, \rho^2 \, \sqrt{e^{\xi/\rho}}. }
In the last step we used that the period of the cosine
function inside the integral is $O(1)$, and hence the 
value of the integral is bounded above by a negative 
constant independent of $r$ and $t$. Moreover, we have
\eqnspl{e:tkappa2-bnd}
{ \tkappa_2
  = d_1 \, \int_0^1 e^{\xi u / \rho} \, u^{1-d/\gamma} \, du
  \le e^{\xi / \rho}, }
which implies that the right hand side of \eqref{e:oscillating}
is at most $-c \, \rho^2 \, \sqrt{\tkappa_2}$. Combining
\eqref{e:decaying}, \eqref{e:oscillating} and \eqref{e:tkappa2-bnd}
yields the statement of the lemma.
\end{proof}

\begin{lemma}
\label{lem:chik-tail-int-both1}
Suppose $\rho \ge \sqrt{k}$. 
Then we have
\eqnst
{ \int_{|t| > \rho/2} |\chi_k(t)| \, dt
  \le \frac{C^k \, k^{k/2}}{\sqrt{\tkappa_2} \, (\sqrt{\tkappa_2} \rho)^{k+1}}
  \quad \text{ and } \quad
  \int_{|t| > \rho/2} \left| \chi_{\tY^{(r)}}(t) \right| \, dt
  \le \frac{C^k \, k^{k/2}}{\sqrt{\tkappa_2} \, (\sqrt{\tkappa_2} \rho)^{k+1}}. }
\end{lemma}

\begin{proof}
Using Lemma \ref{lem:tcjl}, the integral in the first claim
can be estimated by
\eqnspl{e:chik-tail1}
{ \int_{|t| > \rho/2} |\chi_k(t)| \, dt
  &\le 2 \int_{\rho/2}^\infty e^{-\frac{t^2}{2} \tkappa_2} 
      \left[ 1 + \sum_{j=1}^k \frac{\left| {p}_j(it) \right|}{\rho^j} \right] \, dt \\
  &\le \frac{C}{\tkappa_2 \, \rho} e^{-\rho^2 \tkappa_2 / 8} + 
      2 \sum_{j=1}^k \frac{1}{\rho^j} \sum_{\ell=1}^j 
      \frac{\left( \tkappa_2 \right)^\ell \, \ell^{j}}{\ell! \, j!} 
      \int_{\rho/2}^\infty t^{j+2\ell} \, e^{-\frac{t^2}{2} \tkappa_2} \, dt. }
Repeated integration by parts yields
\eqnsplst
{ \int_{\rho/2}^\infty t^{j+2\ell} \, e^{-\frac{t^2}{2} \tkappa_2} \, dt
  = \frac{2 \, e^{- \rho^2 \tkappa_2 / 8}}{\tkappa_2 \, \rho} \, 
    \left[ \left( \rho / 2 \right)^{j+2\ell}
      + \frac{(j+2\ell-1)}{\tkappa_2} \, \left( \rho / 2 \right)^{j+2\ell-3} + \dots \right]. }
Since $j+2\ell-1 \le 3k - 1 \le 3 \, \rho^2 = 12 \, (\rho/2)^2$ and $\tkappa_2 \ge 1$,
the right hand side is bounded above by
\eqnst
{  \frac{C \, e^{- \rho^2 \tkappa_2 / 8}}{\tkappa_2 \, \rho} \, 3k \, \left( \rho / 2 \right)^{j+2\ell}
   \le \frac{C^k \, e^{ - \rho^2 \tkappa_2 / 8 }}{\tkappa_2 \, \rho} \, \frac{\rho^{j+2\ell}}{16^j}. }
Substituting this into the right hand side of \eqref{e:chik-tail1},
we get the upper bound
\eqnsplst
{ \frac{e^{ - \rho^2 \tkappa_2 / 8 }}{\tkappa_2 \, \rho} \, \left[ 2 + C^k \, 
     \sum_{j=1}^k \frac{1}{16^j \, j!}
     \sum_{\ell=1}^j \frac{\left( \tkappa_2 \, \rho^2 \right)^\ell \, \ell^j}{\ell!} \right]. }
Since $\tkappa_2 \, \rho^2 \ge k$, in the second sum the last term is the largest, and hence
we have the bound:
\eqnsplst
{ \frac{e^{ - \rho^2 \tkappa_2 / 8 }}{\tkappa_2 \, \rho} \, \left[ 2 + C^k \, 
    \sum_{j=1}^k \frac{(\tkappa_2 \, \rho^2)^j}{16^j \, j!} \right]
  \le \frac{e^{ - \rho^2 \tkappa_2 / 8 }}{\tkappa_2 \, \rho} \, 
    \left[ 2 + C^k \, e^{ \rho^2 \, \tkappa_2 / 16 } \right]
  \le \frac{C^k \, e^{ - \rho^2 \tkappa_2 / 16 }}{\tkappa_2 \, \rho}. }
The last expression equals
\eqnsplst
{ \frac{C^k \, k^{k/2}}{\sqrt{\tkappa_2} \, (\sqrt{\tkappa_2} \rho)^{k+1}} \, 
    \exp \left( - \frac{k}{2} \, \log k -\frac{\tkappa_2 \, \rho^2}{16} 
    + \frac{k}{2} \log \left( \tkappa_2 \, \rho^2 \right) \right)
  \le \frac{C^k}{\sqrt{\tkappa_2} \, (\sqrt{\tkappa_2} \rho)^{k+1}}, }
using that $\tkappa_2 \rho^2 \ge k$.

For the second claim, the bound derived in \eqref{e:tail-bnd-tYr} 
(Lemma \ref{lem:tail-bnd-tYr}) gives
\eqnspl{e:chYr-int-tail-both}
{ &\int_{|t| > \rho/2} | e^{-iyt} \, \chi_{\tY^{(r)}}(t) | \, dt 
  \le \frac{C}{\rho} \exp \left( - c \, \rho^2 - c \, \rho^2 \, \sqrt{\tkappa_2} \right) 
  \le \frac{C}{\rho} \exp \left( - c \, \rho^2 \, \sqrt{\tkappa_2} \right) \\
  &\qquad = \frac{C \, k^{k/2}}{\sqrt{\tkappa_2} \, (\sqrt{\tkappa_2} \rho)^{k+1}}
    \exp \left( - k \, \frac{\log k}{2} - c \, \rho^2 \, \sqrt{\tkappa_2} + 
    (k+2) \, \log( \sqrt{\tkappa_2} \rho ) \right) \\
  &\qquad \le \frac{C \, k^{k/2}}{\sqrt{\tkappa_2} \, (\sqrt{\tkappa_2} \rho)^{k+1}} \, C^k. }
In the last step we used that $\rho \ge \sqrt{k}$ and $\tkappa_2 \ge 1$.
\end{proof}

Finally, we need the fact that $\tilde{\nf}_k$ is of order $1/\sqrt{\tkappa_2}$, given 
in the following lemma.

\begin{lemma}
\label{lem:right-order}
There exist constants $C_4 \ge 1$, $c$ and $C$ such that if $\rho \ge C_4 \, \sqrt{k}$, then
$c/\sqrt{\tkappa_2} \le \tilde{\nf}_k \le C/\sqrt{\tkappa_2}$.  
\end{lemma}

\begin{proof}
We estimate, using Lemma \ref{lem:tcjl} in the third step: 
\eqnspl{e:nf_k-bnd}
{ &\left| \tnf_k - \frac{1}{\sqrt{2 \pi}} \, \frac{1}{\sqrt{\tkappa_2}} \right|
  \le \int_{-\infty}^{\infty} e^{-t^2 \tkappa_2/2} \, \sum_{j=1}^k \frac{| p_j(it) |}{\rho^j} \, dt \\
  &\qquad \le \sum_{j=1}^k \frac{1}{\rho^j} \sum_{\ell=1}^j \tilde{c}_{j,\ell} \,
       \int_{-\infty}^{\infty} e^{-t^2 \tkappa_2/2} \, |t|^{j+2\ell} \, dt \\
  &\qquad = \frac{1}{\sqrt{\tkappa_2}} \, \sum_{j=1}^k \frac{1}{\rho^j} \, 
       \sum_{\ell=1}^j \frac{\tilde{c}_{j,\ell}}{\tkappa_2^{\frac{j}{2}+\ell}} \,
       2^{\frac{j}{2}+\ell+\frac{1}{2}} \, \Gamma \left( \frac{j}{2} + \ell + \frac{1}{2} \right) \\
  &\qquad \le \frac{1}{\sqrt{\tkappa_2}} \, \sum_{j=1}^k \frac{1}{\rho^j} \, 
       \sum_{\ell=1}^j \frac{\tkappa_2^\ell \, \ell^j}{\ell! \, j! \, \tkappa_2^{\frac{j}{2}+\ell}} \,
       2^{\frac{j}{2}+\ell+\frac{1}{2}} \, \Gamma \left( \frac{j}{2} + \ell + \frac{1}{2} \right) \\
  &\qquad \le \frac{C}{\sqrt{\tkappa_2}} \, \sum_{j=1}^k \frac{1}{(\sqrt{\tkappa_2} \, \rho)^j} \, 
       \sum_{\ell=1}^j \frac{\ell^j \, (j+2\ell)^{\frac{j}{2}+\ell} \, e^{\frac{j}{2}}}      
       {\ell^\ell\, j^j}. }
Writing $\ell = xj$, the summand inside the sum over $\ell$ is of the form:
\eqnsplst
{ \frac{(x j)^j \, \left[ (1 + 2x) j \right]^{(\frac{1}{2}+x) j} \, (\sqrt{e})^j}{(x j)^{x j} \, j^j}
  &= j^{j/2} \, \frac{x^j \, (1 + 2x)^{(\frac{1}{2} + x) j} \, (\sqrt{e})^j}{x^{x j}} \\
  &= j^{j/2} \, \left[ \frac{x \, (1 + 2x)^{\frac{1}{2}+x} \, \sqrt{e}}{x^x} \right]^j 
  \le C^j \, j^{j/2}. }
Therefore, the right hand side of \eqref{e:nf_k-bnd} is at most
\eqnst
{ \frac{C}{\sqrt{\tkappa_2}} \, \sum_{j=1}^k \frac{C^j \, j^{j/2}}{( \sqrt{\tkappa_2} \rho )^j}. }
Choosing $C_4$ large enough the sum over $j$ can be made small, and the statement follows.
\end{proof}

\begin{proof}[Proof of Theorem \ref{thm:main} when 
$y \ge 0$.]
We estimate
\eqnspl{e:pieces}
{ \left| f_{\tY^{(r)}}(y) - \tnf_k \right|
  &\le I_1 + I_2 + I_3 + I_4
      + \frac{1}{2 \pi} \int_{|t| > \rho/2} 
      | \chi_{\tY^{(r)}}(t) | \, dt, }
where 
\eqnsplst
{ I_1
  &:= \frac{1}{2 \pi} \int_{-\rho/2}^{\rho/2}
      \left| e^{\widetilde{u}(t)} - e^{-\frac{t^2}{2} \tkappa_2 + \widetilde{q}_k(t)} \right| \, dt \\
  I_2
  &:= \frac{1}{2 \pi} \int_{-\rho/2}^{\rho/2}
      e^{-\frac{t^2}{2}} \left| e^{\widetilde{q}_k(t)} - 
      \sum_{\ell=0}^k \frac{\widetilde{q}_k(t)^\ell}{\ell!} \right| \, dt \\
  I_3
  &:= \frac{1}{2 \pi} \int_{-\rho/2}^{\rho/2}
      e^{-\frac{t^2}{2} \tkappa_2} \left| 1 + \sum_{j=1}^k \frac{\widetilde{p}_j(it)}{\rho^j}
      - \sum_{\ell=0}^k \frac{\widetilde{q}_k(t)^\ell}{\ell!} \right| \, dt \\
  I_4
  &:= \frac{1}{2 \pi} \int_{|t| > \rho/2} | \chi_k(t) | \, dt. }
The contributions $I_1$, $I_2$, $I_3$, $I_4$, and the last term,
respectively, were estimated in Lemmas \ref{lem:truncate-ut-both}, \ref{lem:truncate-exp-both},
\ref{lem:upto-order-k-both} and \ref{lem:chik-tail-int-both1}, respectively. 
The bounds provided by these lemmas are of the form: $C/\sqrt{\tkappa_2}$ multiplied by
a factor that is, in each case, bounded above by the claimed upper bound on 
$\eps_k$. Due to Lemma \ref{lem:right-order}, $\tnf_k$ is of the order $1/\sqrt{\tkappa_2}$,
and hence the theorem follows.
\end{proof}

\subsection{Lower tail}
\label{sssec:lower-tail-ew}

In this section we prove Theorem \ref{thm:main} in the case 
$-\mu(r)/\sigma(r) < y \le 0$. What makes this case
different from the previous section is that $\tkappa_2 \le 1$,
and in fact, $\tkappa_2 \to 0$ as $y \to -\mu(r)/\sigma(r)$. This means that 
the tail of the Gaussian $e^{-\frac{t^2 \, \tkappa_2}{2}}$
decays slower, and more care is needed.
For ease of notation we write $\zeta = \zeta(y,r) = - \xi(y,r) \ge 0$.
We also write $d_3 = d_1/(1-d/\gamma)$, so that $\mu(r) / \sigma(r) = d_3 \rho$.
Then from \eqref{e:sol-y-2} we have
\eqnspl{e:y-lower-tail}
{ y 
  &= - d_3 \, \rho + d_1 \, \rho \, \int_0^1 e^{-\zeta u / \rho} u^{-d/\gamma} \, du \\
  &= - d_3 \, \rho \left( 1 - \frac{(1 - \frac{d}{\gamma}) \, \rho^{1-d/\gamma}}{\zeta^{1-d/\gamma}} 
    \int_0^{\zeta/\rho} e^{-v} v^{-d/\gamma} \, dv \right), }
which shows that as $y$ varies in the interval $(-d_3 \, \rho, 0]$, the 
number $-\zeta(y,r)$ varies in the interval $(-\infty,0]$. 
It will be useful to write some of the estimates in terms of $\zeta/\rho$, 
rather than $\tkappa_2$. The following relationship will be useful:
\eqnsplst
{ \tkappa_2
  &= d_1 \, \int_0^1 e^{-\zeta u /\rho} \, u^{1-d/\gamma} \, du
  = d_1 \, \left( \frac{\zeta}{\rho} \right)^{-d_1} \, 
    \int_0^{\zeta/\rho} e^{-v} \, v^{d_1-1} \, dv \\
  &= \left( \frac{\zeta}{\rho} \right)^{-d_1} \, \left( \Gamma(d_1+1) 
    - O \left( \left(\frac{\zeta}{\rho}\right)^{d_1-1} \, e^{-\zeta/\rho} \right) \right),
    \quad \text{as $\zeta/\rho \to \infty$.} }
Since $1 = \Gamma(2) < \Gamma(d_1+1) < \Gamma(3) = 2$, we can fix a constant $D$
with the property that 
\eqn{e:D-ass}
{ (\zeta/\rho)^{-d_1} 
  \le \tkappa_2 
  \le 2 (\zeta/\rho)^{-d_1}, \quad \text{ when $\zeta / \rho \ge D$.} }
For later use, we are going to assume that 
\eqn{e:D-ass-2}
{ D \ge 3 \qquad \text{ and } \qquad D^{d/\gamma} \ge 9/2. }

Recall the formula
\eqnspl{e:tchi-expression-ew-lower}
{ e^{-ity} \, \chi_{\tY^{(r)}}(t)
  = \exp \left( d_1 \, \rho^2 \, \int_0^1 e^{-\zeta u / \rho}
    \left( e^{itu/\rho} - 1 - \frac{itu}{\rho} \right) \, u^{-1-d/\gamma} \, du \right), }
where $\rho = r^{d/2}$ and $d_1 = 2 - \frac{d}{\gamma}$.
We will frequently use the observation that 
\eqnspl{e:tkappan-est}
{ \tkappa_n
  &= d_1 \int_0^1 e^{-\zeta u / \rho} \, u^{n-1-d/\gamma} \, du \\
  &\le \tkappa_2 \, (\zeta/\rho)^{-(n-2)} \, \left( n-1-\frac{d}{\gamma} \right) \,
      \left( n - 2 - \frac{d}{\gamma} \right) \cdots \left( 2 - \frac{d}{\gamma} \right) \\
  &\le \tkappa_2 \, (\zeta/\rho)^{-(n-2)} \, (n-1)!, 
      \quad n \ge 2, }
which is obtained via integrating by parts $n-2$ times and dropping negative terms.

We continue to use the expressions defined in \eqref{e:qkrk-etc-ew}.
We give separate arguments depending on whether $\zeta/\rho \ge D$ or not.
In the former case, an argument similar to that in Section \ref{sssec:upper-tail-ew} 
works, and we give a slightly different argument in the latter case.


We first consider $|t| \le \frac{1}{3} \zeta = \frac{1}{3} \, \rho \, (\zeta/\rho)$. 
As in Section \ref{sssec:upper-tail-ew}, we estimate 
$e^{-ity} \, \chi_{\tY^{(r)}}(t) = e^{{u}(t)}$ first by 
${\tchi}_k(t) = \exp \left( - \frac{t^2}{2} \, \tkappa_2 + {q}_k(t) \right)$,
then by the Taylor expansion of $e^{{q}_k(t)}$, and keep terms of order
$\rho^{-k}$ and lower. The following series of lemmas gives the estimates.

\begin{lemma}
\label{lem:qkt-est}
For $|t| < \zeta/3$ we have
\eqn{e:qkt-est}
{ \max \left\{ \left| {q}_k(t) \right|,\, \left| {q}_k(t) + {r}_k(t) \right| \right\} 
  \le \frac{t^2}{2} \, \tkappa_2 \, \frac{|t|}{\zeta} 
  \le \frac{t^2}{4} \, \tkappa_2. }
\end{lemma}

\begin{proof}
From \eqref{e:qkrk-etc-ew} and \eqref{e:tkappan-est} we have that the left hans side of
\eqref{e:qkt-est} is at most
\eqnst
{ \frac{\tkappa_2}{\zeta/\rho} \, \rho^2 \, \sum_{n=3}^{\infty} 
     \frac{|t|^n \, (n-1)!}{(\zeta/\rho)^{n-3} \, \rho^n \, n!}
  \le \tkappa_2 \, \frac{|t|^3}{\zeta} \, \sum_{n=3}^\infty 
     \frac{|t|^{n-3}}{\zeta^{n-3} \, n}
  \le \frac{t^2}{3} \, \tkappa_2 \, \frac{|t|}{\zeta} \, \frac{1}{1 - \frac{|t|}{\zeta}}
  \le \frac{t^2}{4} \, \tkappa_2. }
\end{proof}

\begin{lemma}
\label{lem:truncate-ut-ew-lower}
There exists a constant $C$ such that when $\zeta/\rho \ge D$ then 
for all $r \ge 1$ and $k \ge 1$ we have
\eqnsplst
{ \int_{- \frac{1}{3} \zeta}^{\frac{1}{3} \zeta}
    \left| e^{{u}(t)} - {\tchi}_k(t) \right| \, dt
  \le \frac{C^k \, k^{k/2}}{\sqrt{\tkappa_2}} \, 
    \left(\frac{\tkappa_2^{\frac{1}{d_1}-\frac{1}{2}}}{\rho} \right)^{k+1}. }
When $0 \le \zeta/\rho < D$, the same upper bound applies to
$\int_{-\rho/3}^{\rho/3} \left| e^{{u}(t)} - {\tchi}_k(t) \right| \, dt$.
\end{lemma}

\begin{proof}
As in Lemma \ref{lem:truncate-ut-both}, let $\alpha(t) = - \frac{t^2 \, \tkappa_2}{2} + q_k(t)$,
$\beta(t) = u(t)$ and $x(t) = -\frac{t^2 \, \tkappa_2}{2}$. As in that Lemma, we have
$| \beta(t) - \alpha(t) | \le t^2 \, \frac{|t|^{k+1}}{\rho^{k+1} \, (k+3)!} \, \tkappa_{k+3}$.
Using \eqref{e:tkappan-est}, we get
\eqnsplst
{ | \beta(t) - \alpha(t) | 
  \le t^2 \, \frac{|t|^{k+1} \, (k+2)!}{\rho^{k+1} \, (k+3)!} \, 
    \left( \frac{\zeta}{\rho} \right)^{-(k+1)} \, \tkappa_2 
  \le \frac{t^2}{4} \, \tkappa_2 \, \left( \frac{|t|}{\zeta} \right)^{k+1}
  \le \frac{t^2}{4} \, \tkappa_2. }
Applying \eqref{e:exp-alpha-beta}, using Lemma \ref{lem:qkt-est}, and
then \eqref{e:Gauss-moments} yields: 
\eqnsplst
{ &\int_{- \frac{1}{3} \zeta}^{\frac{1}{3} \zeta}
     \left| e^{-ity} \, \chi_{\tY^{(r)}}(t) - {\tchi}_k(t) \right| \, dt 
  \le \frac{\tkappa_2}{4 \, \zeta^{k+1}} \,
     \int_{-\infty}^\infty |t|^{k+3} \, e^{-t^2 \tkappa_2 /4} \, dt \\
  &\qquad \le \frac{C^k \, k^{k/2}}{((\zeta/\rho) \rho)^{k+1}} \,
     \frac{\tkappa_2}{\tkappa_2^{\frac{k}{2}+2}}
  \le \frac{C^k \, k^{k/2}}{\sqrt{\tkappa_2}} \, 
     \left( \frac{\tkappa_2^{\frac{1}{d_1}-\frac{1}{2}}}{\rho} \right)^{k+1}. }
In the last inequality we used that $(\zeta/\rho) \ge \tkappa_2^{-1/d_1}$ when 
$\zeta / \rho \ge D$ (recall \eqref{e:D-ass}).

When $0 \le \zeta/\rho < D$, the proof of Lemma \ref{lem:truncate-ut-both}
can be followed, noting that $\tkappa_2$ is bounded away from $0$. 
\end{proof}

We now find an estimate for the error resulting from
omitting ${\tilde{r}}_k(t)$ in \eqref{e:tchik-both}. 
We do this in two steps, similarly to Lemmas \ref{lem:truncate-exp-both} and 
\ref{lem:upto-order-k-both}

\begin{lemma}
\label{lem:truncate-exp-ew-lower}
There is a constant $C$ such that if $\zeta/\rho \ge D$ then for all $r \ge 1$,
$k \ge 1$ we have
\eqn{e:exp-diff-ew-lower}
{ \int_{- \frac{1}{3} \zeta}^{\frac{1}{3} \zeta} e^{-t^2 \tkappa_2 / 2} \, 
    \left| e^{{q}_k(t)} 
    - \sum_{\ell=0}^k \frac{{q}_k(t)^\ell}{\ell!} \right| \, dt
  \le \frac{C^k \, k^{k/2}}{\sqrt{\tkappa_2}} \, 
    \left( \frac{\tkappa_2^{\frac{1}{d_1}-\frac{1}{2}}}{\rho} \right)^{k+1}. }
When $0 \le \zeta/\rho < D$, the same upper bound applies to
$\int_{-\rho/3}^{\rho/3} e^{-t^2 \tkappa_2 / 2} \, 
    \left| e^{{q}_k(t)} 
    - \sum_{\ell=0}^k \frac{{q}_k(t)^\ell}{\ell!} \right| \, dt$.
\end{lemma}

\begin{proof}
Using Lemma \ref{lem:qkt-est}, the expression inside absolute values in the left hand 
side of \eqref{e:exp-diff-ew-lower} is at most
\eqnspl{e:qkt-bnd-both2}
{ \frac{\left| {q}_k(t) \right|^{k+1}}{(k+1)!} \, e^{|{q}_k(t)|}
  &\le \frac{\tkappa_2^{k+1}}{\zeta^{k+1}} \, \frac{|t|^{3(k+1)}}{(k+1)!} \, 
      \left( \frac{1}{2} \right)^{k+1} \, e^{t^2 \tkappa_2 / 4} 
  \le \frac{\tkappa_2^{k+1}}{((\zeta / \rho) \, \rho)^{k+1}} \, \frac{C^k}{k^k} \, 
      |t|^{3(k+1)} \, e^{t^2 \tkappa_2 / 4}. }
Using \eqref{e:Gauss-moments} and Stirling's formula, we have
\eqnst
{ \int_{-\zeta/3}^{\zeta/3} |t|^{3(k+1)} e^{-t^2 \tkappa_2 / 4} \, dt
  \le \left( \frac{4}{\tkappa_2} \right)^{\frac{3}{2} k + 2} \, 
      \Gamma \left( \frac{3}{2} k + 2 \right)
  \le C^k \, k^{3k/2} \, \tkappa_2^{- \frac{3}{2} k - 2}. }
Taking into account the additional terms in the right hand side
of \eqref{e:qkt-bnd-both2}, and using the inequality 
$(\zeta/\rho) \ge \tkappa_2^{-1/d_1}$ (when $\zeta / \rho \ge D$), 
this implies that the left hand side 
of \eqref{e:exp-diff-ew-lower} is bounded by the expression claimed 
in the Lemma.

When $0 \le \zeta/\rho < D$, the proof of Lemma \ref{lem:truncate-exp-both} can be followed,
noting that $\tkappa_2$ is bounded away from $0$.
\end{proof}

\begin{lemma}
\label{lem:upto-order-k-ew-lower}
There is a constant $C$ such that if $\zeta/\rho \ge D$ then for any 
$k \ge 1$ and $\rho \ge k$ we have
\eqn{e:ser-diff2}
{ \int_{-\zeta/3}^{\zeta/3} e^{-\frac{t^2}{2} \tkappa_2} \, 
    \left| 1 + \sum_{j=1}^k \frac{{p}_j(it)}{\rho^j} 
    - \sum_{\ell=0}^k \frac{{q}_k(t)^\ell}{\ell!} \right| \, dt
  \le \frac{C^k \, k^{k/2}}{\sqrt{\tkappa_2}} \, 
    \left( \frac{\tkappa_2^{\frac{1}{d_1}-\frac{1}{2}}}{\rho} \right)^{k+1}. }
When $0 \le \zeta/\rho < D$, the same applies to
$\int_{-\rho/3}^{\rho/3} e^{-\frac{t^2}{2} \tkappa_2} \, 
    \left| 1 + \sum_{j=1}^k \frac{{p}_j(it)}{\rho^j} 
    - \sum_{\ell=0}^k \frac{{q}_k(t)^\ell}{\ell!} \right| \, dt$.
\end{lemma}

We will need the following alternative bound on the coefficients $\tilde{c}_{j,\ell}$.

\begin{lemma}
\label{lem:lower-tcjl}
When $-d_3 \, \rho < y \le 0$, we have
\eqnst
{ \tilde{c}_{j,\ell}
  \le \tkappa_2^\ell \, \left( \frac{\zeta}{\rho} \right)^{-j} \,
      \frac{(j + \ell - 1)!}{\ell! \, (\ell-1)! \, j!}. }
\end{lemma}

\begin{proof}
In \eqref{e:tkappan-est} the right hand side is less than 
$\tkappa_2 \, (\zeta/\rho)^{-(n-2)} \, n!$. Using this in the 
definition \eqref{e:tcjl-def} we get:
\eqnsplst
{ \tilde{c}_{j,\ell}
  \le \frac{\tkappa_2^\ell}{\ell!} \, \left( \frac{\zeta}{\rho} \right)^{-j} \, 
     \sum_{\substack{1 \le m_1, \dots, m_\ell \le k \\
     m_1 + \dots + m_\ell = j}} 1
  \le \frac{\tkappa_2^\ell}{\ell!} \, \left( \frac{\zeta}{\rho} \right)^{-j} \, 
     \sum_{\substack{1 \le m_1, \dots, m_\ell \\
     m_1 + \dots + m_\ell = j}} 1
  = \frac{\tkappa_2^\ell}{\ell!} \, \left( \frac{\zeta}{\rho} \right)^{-j} \,
    \binom{j + \ell - 1}{\ell - 1}. }
\end{proof}

\begin{proof}[Proof of Lemma \ref{lem:upto-order-k-ew-lower}.]
(i) The statement is vacuous when $k = 1$, so assume $k \ge 2$.
We estimate the expression inside absolute values as in Lemma \ref{lem:upto-order-k-both}, 
this time using Lemma \ref{lem:lower-tcjl} to bound $\widetilde{c}_{j,\ell}$. 
This yields the following bound on the left hand side of \eqref{e:ser-diff2}:
\eqnspl{e:sum-j-ell-ew-lower}
{ &\frac{1}{\sqrt{\tkappa_2}} \, 
    \sum_{j=k+1}^{k^2} \frac{(\zeta/\rho)^{-j}}{(\sqrt{\tkappa_2} \rho)^j} \, 
    \sum_{\ell = \lceil j/k \rceil}^{k}
    \frac{2^{\frac{j}{2} + \ell + \frac{1}{2}} \,
    \Gamma(\frac{j}{2} + \ell + \frac{1}{2}) \, (j + \ell - 1)!}{\ell! \, (\ell-1)! \, j!}. }
The largest term in the second sum is for $\ell = k$, and hence we can bound above 
the sum over $\ell$ by $k$ times the $\ell = k$ term. 
By Stirling's formula \eqref{e:Stirling}, the expression
in \eqref{e:sum-j-ell-ew-lower} is at most
\eqnsplst
{ &\frac{C^k}{\sqrt{\tkappa_2}}
  \sum_{j=k+1}^{k^2} \frac{(\zeta/\rho)^{-j}}{(\sqrt{\tkappa_2} \rho)^j} \,
  \frac{ 2^{\frac{j}{2} + k} \, 
    (\frac{j}{2} + k)^{\frac{j}{2} + k} \,
    e^{-\frac{j}{2} - k} \, 
    (j + k)^k }{ k^k \, k^k } \\
  &\qquad \le \frac{C^k}{\sqrt{\tkappa_2}}
  \sum_{j=k+1}^{k^2} \frac{(\zeta/\rho)^{-j}}{(\sqrt{\tkappa_2} \rho)^j} \,
  \frac{ 2^{\frac{j}{2}} \, 
    (\frac{j}{2} + k)^{\frac{j}{2} + k} \,
    e^{-\frac{j}{2}} \, 
    (j + 2k)^k }{ k^k \, k^k }. }
Writing $j = x k$, the right hand side equals:
\eqnspl{e:sum-j-lower-2}
{ &\frac{C^k}{\sqrt{\tkappa_2}}
  \sum_{j=k+1}^{k^2} \frac{(\zeta/\rho)^{-j}}{(\sqrt{\tkappa_2} \rho)^j} \,
  \frac{ \left[ (x+2) \, k \right]^{(\frac{x}{2}+1) k} \,
    e^{-(\frac{x}{2}+1) k} \, \left[ (x+1) \, k \right]^k }{k^{2k}}  \\
  &\qquad = \frac{C^k}{\sqrt{\tkappa_2}}
  \sum_{j=k+1}^{k^2} \frac{(\zeta/\rho)^{-j} k^{\frac{j}{2}}}{(\sqrt{\tkappa_2} \rho)^j} \,
  \left[ (x+2)^{\frac{x}{2}+1} \, e^{-(\frac{x}{2}+1)} \, (x+1) \right]^k. }
We use now that $\tkappa_2 \ge (\zeta/\rho)^{-d_1}$ when $\zeta/\rho \ge D$ and hence 
that $1/(\sqrt{\tkappa_2})^j \le (\zeta/\rho)^{\frac{d_1}{2} j}$.
We also use that
\eqnst
{ \frac{1}{D^{2-d_1}} \left( \frac{\zeta}{\rho} \right)^{2 - d_1} \rho^2
  \ge \rho^2 
  \ge k^2
  \ge j
  = x k, }
which implies that 
\eqnst
{ \frac{1}{\rho}
  \le \frac{1}{D^{1-\frac{d_1}{2}}} \, 
      \left( \frac{\zeta}{\rho} \right)^{1-\frac{d_1}{2}} \, 
      \frac{1}{x^{1/2} k^{1/2}}. }
From this 
\eqnsplst
{ \frac{1}{(\sqrt{\tkappa_2} \rho)^j}
  &\le \frac{(\zeta/\rho)^{\frac{d_1}{2} j}}{\rho^{k+1}} \, \frac{1}{\rho^{j - (k+1)}} 
  \le \frac{(\zeta/\rho)^{\frac{d_1}{2} j}}{\rho^{k+1}} \,
      \frac{1}{D^{(1-\frac{d_1}{2})(j-k-1)}} \, 
      \left( \frac{\zeta}{\rho} \right)^{(1-\frac{d_1}{2}) (j-k-1)} \,
      \frac{k^{-\frac{1}{2}(j-k-1)}}{x^{\frac{1}{2} (j-k-1)}} \\
  &\le \frac{(\zeta/\rho)^{j}}{\rho^{k+1}} \,
      \frac{D^{1-\frac{d_1}{2}}}{D^{(1-\frac{d_1}{2}) (x-1) k}} \,
      \left( \frac{\zeta}{\rho} \right)^{-(1-\frac{d_1}{2})(k+1)} \,
      \frac{x^{1/2}}{x^{\frac{1}{2} (x-1) k}} \,
      k^{-\frac{j}{2}} \, k^{\frac{k+1}{2}}. }
Substituting this into the right hand side of \eqref{e:sum-j-lower-2} we get that
the expression in \eqref{e:sum-j-lower-2} is at most:
\eqnsplst
{ \frac{D^{1-\frac{d_1}{2}} \, C^k \, k^{k/2} \, 
     (\zeta/\rho)^{-(1-\frac{d_1}{2}) (k+1)}}{\sqrt{\tkappa_2} \, \rho^{k+1}}
  \sum_{j=k+1}^{k^2} 
      \left[ \frac{(x+2)^{\frac{x}{2}+1} \, e^{-(\frac{x}{2}+1)} \, (x+1)}
      {\left( D^{1-\frac{d_1}{2}} \right)^{x-1} \, x^{\frac{x-1}{2}}} \right]^k. }
Choosing $D$ large enough, the expression inside square brackets is at most a
constant $C$, and hence the sum over $j$ is bounded by $k^2 C^k$.
Noting that $(\zeta/\rho)^{-(1-\frac{d_1}{2})} \le \tkappa_2^{\frac{1}{d_1}-\frac{1}{2}}$ 
when $\zeta/\rho \ge D$ completes the proof of the first statement.

When $0 \le \zeta/\rho < D$, the proof of Lemma \ref{lem:upto-order-k-both}
can be followed, noting that $\tkappa_2$ is bounded away from $0$.
\end{proof}

We now prove bounds on the tails of $\chi_k(t)$ and $e^{-ity} \chi_{\tY^{(r)}}(t)$.

\begin{lemma}
\label{lem:tchi-tail-tkappa2}
When $\zeta/\rho \ge D$, we have
\eqn{e:tchi-tail-tkappa2-2}
{ \left| e^{-iyt} \, \chi_{\tY^{(r)}}(t) \right| 
  \le \exp \left( -c \, \rho^{d_1} \, |t|^{d/\gamma} \right). }
\end{lemma}

\begin{proof}
We consider the range $|t| \ge \frac{1}{3} \zeta = \frac{1}{3} \, \rho \, (\zeta/\rho)$. 
Recall that $D \ge 3$, and hence $\rho \, |t|^{-1} \le 3 (\zeta/\rho)^{-1} \le 1$.
When $0 \le u \le \rho \, |t|^{-1}$, we have 
$|t \, u / \rho| \le 1$, and 
$\zeta \, u / \rho \le (\zeta / \rho) \, \rho \, |t|^{-1} \le 3$. Therefore,
\eqnspl{e:tchi-re-log-estimate}
{ &\Re \, \log \left( e^{-ity} \, \chi_{\tY^{(r)}}(t) \right)
  = d_1 \, \rho^2 \, \int_0^1 e^{-\zeta u / \rho} \,
    \left( \cos \left( \frac{t u}{\rho} \right) - 1 \right) \, u^{-1-d/\gamma} \, du \\
  &\qquad \le d_1 \, \rho^2 \, \int_0^{\rho |t|^{-1}} e^{-3} \,
    \left( - \frac{t^2 u^2}{4 \rho^2} \right) \, u^{-1-d/\gamma} \, du 
  = - c \, t^2 \, \int_0^{\rho |t|^{-1}} u^{1-d/\gamma} \, du \\
  &\qquad = - c \, t^2 \, \rho^{d_1} \, |t|^{-d_1}
  = - c \, \rho^{d_1} \, |t|^{\frac{d}{\gamma}}. }
This proves the claim.
\end{proof}

\begin{lemma}
\label{lem:chik-tail-int-ew-lower}
Suppose $\rho \ge k \ge 1$. If $\zeta/\rho \ge D$ then we have
\eqnst
{ \int_{|t| > \zeta/3} |\chi_k(t)| \, dt
  \le \frac{C^k \, k^{k/2} \, (\tkappa_2^{\frac{1}{d_1}-\frac{1}{2}})^{k+1}}{\sqrt{\tkappa_2} \, 
      \rho^{k+1}} \quad \text{and} \quad
  \int_{|t| > \zeta/3} |\chi_{\tY^{(r)}}(t)| \, dt
  \le \frac{C^k \, k^{k/2} \, (\tkappa_2^{\frac{1}{d_1}-\frac{1}{2}})^{k+1}}{\sqrt{\tkappa_2} \, 
      \rho^{k+1}}. }
When $0 \le \zeta/\rho < D$, the same bounds hold for the intergal over $\{ |t| > \rho/3 \}$.
\end{lemma}

\begin{proof}
Using Lemma \ref{lem:lower-tcjl}, the integral in the first claim 
can be bounded above by
\eqnspl{e:chik-tail}
{ \int_{|t| > \zeta/3} |\chi_k(t)| \, dt
  &\le 2 \int_{\zeta/3}^\infty e^{-\frac{t^2}{2} \tkappa_2} 
      \left[ 1 + \sum_{j=1}^k \frac{1}{\rho^j} \,
      \sum_{\ell=1}^j \frac{\tkappa_2^\ell \, (\zeta/\rho)^{-j} \, 
      \binom{j+\ell-1}{\ell-1}}{\ell!} \, t^{j+2\ell} \right] \, dt. }
As in Lemma \ref{lem:chik-tail-int-both1} we have
\eqnspl{e:integr-by-parts}
{ \int_{\zeta/3}^\infty t^{j+2\ell} \, e^{-\frac{t^2}{2} \tkappa_2} \, dt
  &= \frac{e^{- \zeta^2 \tkappa_2 / 18}}{\tkappa_2} \, 
    \left[ \left( \frac{\zeta}{3} \right)^{j+2\ell-1}
      + \frac{j+2\ell-1}{\tkappa_2} \, \left( \frac{\zeta}{3} \right)^{j+2\ell-3} 
      + \dots \right]. }
Note that $j+2\ell-1 \le 3k$ and if $D^{d/\gamma} \ge 9/2$, we have:
\eqnst
{ \tkappa_2 (\zeta/3)^2
  \ge \frac{1}{9} \, (\zeta/\rho)^{-d_1} \, (\zeta/\rho)^2 \, \rho^2
  \ge \frac{1}{9} \, (\zeta/\rho)^{d/\gamma} \, \rho^2
  \ge \frac{D^{d/\gamma}}{9} \rho^2 
  \ge k^2/2 
  = (1/6) (3k). }
This implies that the sum in square brackets in \eqref{e:integr-by-parts} 
is bounded above by $C^k (\zeta/3)^{j+2\ell-1}$.
Observe that this bound is also valid when $j = \ell = 0$.
Substituting this into the right hand side of \eqref{e:chik-tail},
we get the upper bound
\eqnspl{e:sum-j-ell-lower}
{ &\frac{C^k \, e^{ - \zeta^2 \tkappa_2 / 18 }}{\tkappa_2 \, \zeta} \, 
    \left[ 1 + \sum_{j=1}^k \frac{1}{\rho^j}
     \sum_{\ell=1}^j \frac{\tkappa_2^\ell \, (\zeta/\rho)^{-j} \, 
     \binom{j+\ell-1}{\ell-1} }{\ell!} \, 
     \frac{\zeta^j \, (\zeta^2)^\ell}{3^{j+2\ell}} \right] \\
  &\qquad = \frac{C^k \, e^{ - \zeta^2 \tkappa_2 / 18 }}{\tkappa_2 \, \zeta} \, 
    \left[ 1 + \sum_{j=1}^k \frac{1}{3^j}
     \sum_{\ell=1}^j \frac{ \left( \frac{1}{9} \tkappa_2 \zeta^2 \right)^\ell \, 
     \binom{j+\ell-1}{\ell-1} }{\ell!} \right]. }
In the sum over $\ell$, the ratio of the $\ell+1$-st and $\ell$-th terms equals:
\eqnst
{ \frac{1}{9} \, \tkappa_2 \, \zeta^2 \, \frac{j+\ell}{\ell} \, \frac{1}{\ell+1}
  \ge \frac{k^2}{2} \, \frac{2 \ell}{\ell} \, \frac{1}{\ell+1}
  \ge k 
  \ge 1. }
Hence the sum over $\ell$ in \eqref{e:sum-j-ell-lower} is at most $k$ times the last term,
and hence the right hand side of \eqref{e:sum-j-ell-lower} is at most
\eqnspl{e:sum-j-lower}
{ &\frac{C^k \, e^{ - \zeta^2 \tkappa_2 / 18 }}{\tkappa_2 \, \zeta} \, 
    \left[ 1 + \sum_{j=1}^k \frac{\left( \frac{1}{9} \tkappa_2 \zeta^2 \right)^j \,
     \binom{2j-1}{j-1}}{3^j \, j!} \right] 
  \le \frac{C^k \, e^{ - \zeta^2 \tkappa_2 / 18 }}{\tkappa_2 \, \zeta} \, 8^k \, 
    \left[ 1 + \sum_{j=1}^k \frac{\left( \frac{1}{36} \tkappa_2 \zeta^2 \right)^j}{j!} 
    \right] \\
  &\qquad \le \frac{C^k \, e^{ - \zeta^2 \tkappa_2 / 36 }}{\tkappa_2 \, \zeta} 
  = \frac{C^k}{\sqrt{\tkappa_2} \, \sqrt{\tkappa_2} \zeta} \, 
    e^{ - \zeta^2 \tkappa_2 / 16 }. }
Since 
\eqnst
{ \sqrt{\tkappa_2} \, \zeta 
  = \sqrt{\tkappa_2} \, (\zeta/\rho) \, \rho
  \ge (\zeta/\rho)^{1-\frac{d_1}{2}} \, \rho
  \ge \tkappa_2^{-\frac{1}{d_1}(1 - \frac{d_1}{2})} \, \rho 
  = \frac{\rho}{\tkappa_2^{\frac{1}{d_1} - \frac{1}{2}}}, }
the right hand side of \eqref{e:sum-j-lower} is at most
\eqnsplst
{ \frac{C^k}{\sqrt{\tkappa_2}} \, 
     \left( \frac{\tkappa_2^{\frac{1}{d_1}-\frac{1}{2}}}{\rho} \right)^{k+1} \,
     \exp \left( - \frac{1}{36} \, \left( \frac{\zeta}{\rho} \right)^{2-d_1} \, \rho^2 + k \, 
     \log \left( (\zeta/\rho)^{1-\frac{d_1}{2}} \, \rho \right) \right). }
Using that $(\zeta/\rho)^{1-\frac{d_1}{2}} \, \rho \ge D^{1-\frac{d_1}{2}} \, k$, 
the exponential is at most $C^k$, thus we get the claimed bound for $\chi_k(t)$ when 
$\zeta/\rho \ge D$.

When $0 \le \zeta/\rho < D$, the proof of Lemma \ref{lem:chik-tail-int-both1} can be
followed, noting that $\tkappa_2$ is bounded away from $0$.

Finally, for the estimates on $\chi_{\tY^{(r)}}(t)$ we can integrate the bound
in Lemma \ref{lem:tchi-tail-tkappa2}, which gives
\eqnsplst
{ &\int_{|t| > \zeta/3} \left| e^{-iyt} \, \chi_{\tY^{(r)}}(t) \right| \, dt
  \le C \, \int_{\frac{1}{3} \, \rho \, (\zeta/\rho)}^\infty
     \exp \left( - c \, \rho^{d_1} \, |t|^{d/\gamma} \right) \, dt  \\
  &\qquad = \frac{C}{\rho^{1-2\gamma/d}} \, \int_{c \, \rho^2 \, (\zeta/\rho)^{d/\gamma}}^\infty
     v^{\frac{\gamma}{d}-1} \, e^{-v} \, dv \\
  &\qquad \le \frac{C}{\rho^{1-2\gamma/d}} \, \rho^{2\gamma/d-2} \, 
     \left( \frac{\zeta}{\rho} \right)^{1-d/\gamma} \, 
     \exp \left( - c \, \rho^2 \, (\zeta/\rho)^{d/\gamma} \right) \\
  &\qquad = \frac{C}{\rho} \, \left( \frac{\zeta}{\rho} \right)^{d_1} \,
      \left( \frac{\zeta}{\rho} \right)^{-1} \,
      \rho^{\frac{4 \gamma}{d} - 2} \, 
      \exp \left( -c \, \rho^2 \, \left( \zeta/\rho \right)^{d/\gamma} \right) \\
  &\qquad \le \frac{C}{\sqrt{\tkappa_2}} \, \frac{\tkappa_2^{\frac{1}{d_1}-\frac{1}{2}}}{\rho} \,
      \exp \left( -c \, \rho^2 \, \left( \zeta/\rho \right)^{d/\gamma} 
      + \left( \frac{4 \gamma}{d} - 2 \right) \, \log \rho \right) \\
  &\qquad = \frac{C}{\sqrt{\tkappa_2}} \, \frac{\tkappa_2^{\frac{1}{d_1}-\frac{1}{2}}}{\rho} \,
      \left( \frac{1}{\rho \, (\zeta/\rho)^{d/2\gamma}} \right)^k \\
  &\qquad\quad  \times \exp \left( -c \, \rho^2 \, \left( \zeta/\rho \right)^{d/\gamma} 
      + k \, \log \left( \rho \, (\zeta/\rho)^{d/2\gamma} \right)
      + \left( \frac{4 \gamma}{d} - 2 \right) \, \log \rho \right) \\
  &\qquad \le \frac{C}{\sqrt{\tkappa_2}} \, 
      \left( \frac{\tkappa_2^{\frac{1}{d_1}-\frac{1}{2}}}{\rho} \right)^{k+1} \,
      \exp \left( -c \, \rho^2 \, \left( \zeta/\rho \right)^{d/\gamma} 
      + k \, \log \left( \rho \, (\zeta/\rho)^{d/2\gamma} \right)
      + \left( \frac{4 \gamma}{d} - 2 \right) \, \log \rho \right). }
Apart from the $\log \rho$ term in the exponential, this is the same expression as what was 
estimated above, and it satisfies the same bound.
When $0 \le (\zeta/\rho) < D$, we can argue as in the proof of Lemma \ref{lem:chik-tail-int-both1}.
\end{proof}

Finally, we need an analogue of Lemma \ref{lem:right-order} to show that $\tnf_k$ 
is of order $1/\sqrt{\tkappa_2}$.

\begin{lemma}
\label{lem:right-order-lower}
There exist constants $C_4$ and $C$, $c$ such that when $y \le 0$, and
$\rho \ge C_4 k$, we have
$c/\sqrt{\tkappa_2} \le \tnf_k \le C/\sqrt{\tkappa_2}$.
\end{lemma}

\begin{proof}
Using Lemma \ref{lem:lower-tcjl}, we estimate:
\eqnsplst
{ &\left| \tnf_k - \frac{1}{\sqrt{2\pi}} \frac{1}{\sqrt{\tkappa_2}} \right|
  \le \frac{1}{\tkappa_2} \sum_{j=1}^k \frac{1}{\rho^j} \sum_{\ell=1}^j
      \frac{\tilde{c}_{j,\ell}}{\tkappa_2^{\frac{j}{2}+\ell}} \,
      2^{\frac{j}{2}+\ell+\frac{1}{2}} \, \Gamma \left( \frac{j}{2} + \ell + \frac{1}{2} \right) \\
  &\qquad \le \frac{1}{\tkappa_2} \sum_{j=1}^k 
      \frac{C^j \, (\zeta/\rho)^{-j}}{\tkappa_2^{\frac{j}{2}} \, \rho^j} 
      \sum_{\ell=1}^j \frac{(j+\ell-1)! \, \Gamma \left( \frac{j}{2} + \ell + \frac{1}{2} \right)}
      {\ell! \, (\ell-1)! \, j!} \\
  &\qquad \le \frac{1}{\tkappa_2} \sum_{j=1}^k 
      \frac{C^j \, (\zeta/\rho)^{-j}}{\tkappa_2^{\frac{j}{2}} \, \rho^j} 
      \sum_{\ell=1}^j \frac{\Gamma \left( \frac{j}{2} + \ell + \frac{1}{2} \right)}{\ell!}. }
The largest term in the sum over $\ell$ is for $\ell = j$, and hence, using 
Stirling's formula, the right hand side is at most
\eqnsplst
{ \frac{1}{\tkappa_2} \sum_{j=1}^k 
      \frac{C^j \, j^{j/2} \, (\zeta/\rho)^{-j}}{\tkappa_2^{\frac{j}{2}} \, \rho^j}. }
Now if $\zeta/\rho \ge D$, we have $(\zeta/\rho)^{-1} \le \tkappa_2^{1/d_1}$, and hence 
we have the upper bound
\eqnsplst
{ \frac{1}{\tkappa_2} \sum_{j=1}^k 
      \frac{C^j \, j^{j/2} \, \tkappa_2^{(\frac{1}{d_1}-\frac{1}{2}) j}}{\rho^j}
  \le \frac{1}{\tkappa_2} \sum_{j=1}^k 
      \frac{C^j \, j^{j/2}}{\rho^j}, }
and the sum can be made small by choosing $C_4$ large.

If $0 \le \zeta/\rho < D$, we reach the same conclusion, since $\tkappa_2$ is then 
bounded away from $0$.
\end{proof}

\begin{proof}[Proof of Theorem \ref{thm:main} when $y \le 0$.]
The proof is similar to the computations in Section \ref{sssec:upper-tail-ew}.
Combining the estimates in Lemmas \ref{lem:truncate-ut-ew-lower}, 
\ref{lem:truncate-exp-ew-lower}, \ref{lem:upto-order-k-ew-lower} and 
\ref{lem:chik-tail-int-ew-lower}, yields the statement.
This completes the proof of the Theorem.
\end{proof}

\section{Stochastic comparison and error of the approximation}
\label{sec:stoch-mon}

The goal of this section is to prove Theorem \ref{thm:combined}
building on the main technical estimate of Theorem \ref{thm:main}. 
Uniformity in the conditioning will be 
achieved via the following bound on the escape rate of 
$R_\ell$, that follows easily from Theorem \ref{thm:stoch-mon}.
Since the proof of Theorem \ref{thm:stoch-mon} itself is somewhat 
technical, we defer it to the end of this section.

Recall that we fix $\lambda = (2 \gamma - d)/\omega_{d-1}$.
It will be convenient to abreviate 
$d_2 := \frac{d_1}{d/\gamma} = 2 \frac{\gamma}{d} - 1 = \lambda \omega_{d-1} / d$.

\begin{proposition}
\label{prop:incr-R's}
For any $0 < a < 1/d_2$ there exist constants $C = C(a)$ and 
$c = c(a)$ such that for all $0 < s < \infty$ and 
for all $\ell \ge 2$ we have
\eqnst
{ \prob \big( R_{\ell} \le (a \ell)^{1/d} \,\big|\, S = s \big)
  \le C \exp \left( - c \ell \right). }
\end{proposition}

\begin{proof}[Proof of Proposition \ref{prop:incr-R's} --- 
assuming Theorem \ref{thm:stoch-mon}.]
Consider the coupling stated in Theorem \ref{thm:stoch-mon} with 
unconditioned radii $R'_1, R'_2, \dots$. Under this coupling we have
\eqnspl{e:escape-rate}
{ \prob \big( R_{\ell} \le (a \ell)^{1/d} \,\big|\, S = s \big)
  \le \prob \big( R'_{\ell-1} \le (a \ell)^{1/d} \big). }
Since $\left( R'_{\ell-1} \right)^d \stackrel{\mathrm{d}}{=} T_1 + \dots + T_{\ell-1}$,
where $T_1, T_2, \dots$ are i.i.d.~$\mathsf{Exp}(d_2)$ random variables,
for any $\theta > 0$ the right hand side of \eqref{e:escape-rate} equals
\eqnst
{ \prob \left( \sum_{j=1}^{\ell-1} T_j \le a \ell \right)
  = \prob \left( \prod_{j=1}^{\ell-1} e^{-\theta T_j} 
    \ge e^{-\theta a \ell} \right) 
  \le e^{\theta a} \left( \frac{d_2 \, e^{a \theta}}{\theta + d_2} \right)^{\ell-1}. }
This satisfies the claimed bound, when $\theta$ is chosen 
sufficiently small.
\end{proof}

\begin{proof}[Proof of Theorem \ref{thm:combined} --- assuming
Theorem \ref{thm:stoch-mon}.]
Recall that $a_0$ is a number such that $0 < a_0 < 1/d_2$. Take $a = a_0$ 
in Proposition \ref{prop:incr-R's}, and 
let $k = \lfloor \sqrt{a_0 \, \ell} \rfloor$. The difference in the left 
hand side of \eqref{e:diff-DF} can be written as an integral over 
the variables $(r_1, \dots, r_\ell)$. We split the domain of 
integration into two regions:
\eqnsplst
{ \caR
  &= \left\{ (r_1, \dots, r_\ell) : s^{-1/\gamma} < r_1 \le r,\, 
     r_1 < \dots < r_\ell \right\} 
  = \caR_1 \cup \caR_2 \\
  \caR_1
  &= \left\{ (r_1, \dots, r_\ell) : s^{-1/\gamma} < r_1 \le r,\, 
     r_1 < \dots < r_\ell
     \le (a_0 \ell)^{1/d} \right\} \\
  \caR_2
  &= \left\{ (r_1, \dots, r_\ell) : s^{-1/\gamma} < r_1 \le r,\, 
     r_1 < \dots < r_\ell,\, 
     r_\ell > (a_0 \ell)^{1/d} \right\}. }
We then have
\eqnsplst
{ &\prob [ R_1 \le r \,|\, S = s ]
  = \frac{1}{f_S(s)} \int_\caR dr_1 \dots dr_\ell \, 
    f_{R_1,\dots,R_\ell}(r_1,\dots,r_\ell) \, 
    f_{\oS^{(r_\ell)}}(s - r_1^{-\gamma} - \dots - r_\ell^{-\gamma}) \\
  &\qquad = \frac{1}{f_S(s)} \int_{\caR_2} \ + \, \frac{1}{f_S(s)} \int_{\caR_1} \ 
  =: I_2 + I_1. }
Due to Proposition \ref{prop:incr-R's}, we have
\eqnst
{ I_1
  \le \prob [ R_\ell \le (a_0 \ell)^{1/d} \,|\, S = s ]
  \le C \, e^{-c \ell}. }
On $\caR_2$, we have $r_\ell^{d/2} > \sqrt{a_0 \ell} \ge k$, 
and requiring $\ell \ge C_4^4 \, a_0^{-1}$ ensures that 
also $r_\ell^{d/2} \ge C_4 \, \sqrt{k}$. Hence on $\caR_2$ 
Theorem \ref{thm:main} can be applied, and we can write 
\eqnsplst
{ f_{\oS^{(r_\ell)}} (s - r_1^{-\gamma} - \dots - r_\ell^{-\gamma})
  &= \frac{1}{\sigma(r_\ell)} f_{Y^{(r_\ell)}}(y)
  = \frac{1}{\sigma(r_\ell)} \tnf_k e^{-\xi y} \phi_{Y^{(r_\ell)}} (-\xi) \,
    [ 1 + \eps_k(y,r_\ell) ] \\
  &= \fg_{\ell,k}(y,r_\ell) \, [ 1 + \eps_k(y,r_\ell) ], }
where
\eqnsplst
{ |\eps_k(y,r_\ell)|
  \le \frac{C_2 \, C_3^k \, k^{k/2}}{(r_\ell^{d/2})^{k+1}}
  \le \frac{C_2 \, C_3^k \, k^{k/2}}{\sqrt{a_0 \ell}^{\sqrt{a_0 \ell}}}
  \le C \, \exp ( c \, \sqrt{\ell} - c' \, \sqrt{\ell} \, \log \ell ). }
If $\ell$ is sufficiently large, we have $|\eps_k| \le 1/2$, and hence
\eqnsplst
{ f_{\oS^{(r_\ell)}} (s - r_1^{-\gamma} - \dots - r_\ell^{-\gamma})
  = \fg_{\ell,k}(y,r_\ell) + O \left( e^{-c \sqrt{\ell} \, \log \ell} \right) \,
    f_{\oS^{(r_\ell)}} (s - r_1^{-\gamma} - \dots - r_\ell^{-\gamma}). }
This implies that 
\eqnst
{ I_2
  = \frac{1}{f_S(s)} \, \int_{s^{-1/\gamma}}^r \ff_{R_1,S}^{\ell,k}(r_1, s) \, dr_1
    + O \left( e^{-c \sqrt{\ell} \, \log \ell} \right). } 
This completes the proof of the theorem.
\end{proof}

In the proof of Proposition \ref{prop:incr-R's} we used the 
stochastic comparison stated in Theorem \ref{thm:stoch-mon}, which we
now prove. The proof has two steps: we first prove a version 
with a fixed number of points in a finite region, stated in
the next proposition. In the second step we pass to the limit
of infinitely many points.

Let $U_1, \dots, U_n$ be i.i.d.~$\Unif(0,m)$ random variables
($m > 0$) with order statistics $U^{(1)} < \dots < U^{(n)}$. Let 
$W_j = U_j^{-\gamma/d}$, where $\gamma > d$. Let also 
$U'_1, \dots, U'_{n-1}$ be i.i.d.~$\Unif(0,m)$ with order
statistics $U^{(1)'} < \dots < U^{(n-1)'}$.

\begin{proposition}
\label{prop:finite-coupling}
There is a coupling between  
\eqnst
{ \text{the conditional law of $(U_1, \dots, U_n)$ given $\sum_{j=1}^n W_j = s$ 
    ($s > n \, m^{-\gamma/d}$)} } 
and 
\eqnst
{ \text{the unconditional law of $U'_1, \dots, U'_{n-1}$} } 
such that we have $U^{(j)} \ge U^{(j-1)'}$, $2 \le j \le n$, almost surely. 
\end{proposition}

\begin{proof}
The proof has two parts. We first prove the case $n = 2$,
and then use it to prove the general case via a Gibbs sampling argument.

\medbreak

\emph{The case $n = 2$.} 
Let $\Sigma_1$ be the set
\eqnst{
  \Sigma_1
  = \left\{ (u_1, u_2) \in (0,m)^2 : 
    \text{$u_1 \le u_2$ and
    $u_1^{-\gamma/d} + u_2^{-\gamma/d} = s$} \right\}. }
Let $s_0 = s - m^{-\gamma/d}$, and let   
\eqnst
{ \Sigma_0 
  = \left\{ (u_1, u_2) \in (0,m)^2 : \text{$u_1 = s_0^{-d/\gamma}$ and 
    $u_2 \ge s_0^{-d/\gamma}$} \right\}. }
Let $P$ be the orthogonal projection onto the line 
$u_1 = s_0^{-d/\gamma}$, and $\Sigma'_1 = P(\Sigma_1)$. 
Observe that $\Sigma'_1 \subset \Sigma_0$.

When $U_1 = \min \{ U_1, U_2 \}$, the conditional distribution 
of $(U_1, U_2)$ given $W_1 + W_2 = s$, equals the 
normalized length measure on $\Sigma_1$. Expressing
$u_1 = h(u_2) := \left( s - u_2^{-\gamma/d} \right)^{-d/\gamma}$,
this normalized length measure is mapped by $P$ to  
$d\mu_1 = \frac{1}{Z_1} \, f(u_2) \, du_2$,
where $f(u_2) = \sqrt{ 1 + \left( \frac{dh}{du_2} ) \right)^2 }$, 
and $Z_1$ is a normalization constant. 
On the other hand, normalized length measure on $\Sigma_0$ equals
$\mu_0 = \frac{1}{Z_0} \, du_2$, where
$Z_0 = m - s_0^{-d/\gamma}$.

Consider a transformation $T : \Sigma'_1 \to \Sigma_0$  
of the form $T(u_2) = u_2 - g(u_2)$, with a positive function $g = g(u_2)$. 
The Jacobian of $T$ equals $1 - \frac{dg}{du_2}$, and hence, if 
$1 - \frac{dg}{du_2} = f = \sqrt{1 + (\frac{dh}{du_2})^2}$, then 
the image of $\mu_1$ under $T$ is a uniform measure on the
set $[ m-u_* - g(m-u_*), m ]$, where $u_*$ is defined as the
solution to $m - u_* = h(m - u_*)$. Hence we take $g$ to be the function
\eqnst
{ g(m-u)
  = \int_0^u \left[ \sqrt{1 + \left( h'(m-v) \right)^2} - 1 \right] \, dv. }
We show that the image of $T$ is a subset of $[h(m), m]$.
Indeed,
\eqnst
{ g(m-u_*) 
  \le \int_0^{u_*} | h'(m-v) | \, dv
  = - \int_0^{u_*} h'(m-v) \, dv
  = h(m) - h(m-u_*)
  = h(m) - (m - u_*). }
Thus we have shown that the conditional distribution 
of $U_2$, given $W_1 + W_2 = s$ and $U_1 \le U_2$ 
is stochastically larger than that of
a random variable $U'_2$ that is uniform on
$[ s_0^{-d/\gamma}, m ]$. This is stronger than what was
required for the statement. The map $T$ composed with a 
linear stretch of the interval $[s_0^{-2/\gamma}, m ]$ to
$[0,m]$ defines a coupling of the two distributions considered,
that we shall use in the sequel.

\medbreak

\emph{The case $n \ge 3$.} We are going to construct a sequence 
$(\mathbf{U}_k, \mathbf{U}'_k)_{k \ge 0}$ of pairs of random vectors
with the following properties:\\
(i) for each $k \ge 0$, the order statistics of 
$\mathbf{U}_k = (U_{k,1}, \dots, U_{k,n})$ is that of $n$ 
i.i.d.~$\mathsf{Unif}(0,m)$ r.v.~conditioned on 
$\sum_{j=1}^n U_{k,j}^{-\gamma/d} = s$.\\
(ii) for each $k \ge 0$ we have $\mathbf{U}'_k \le \mathbf{U}_k$
componentwise.\\
(iii) almost surely, for all large enough $k$, we have that
there exists exactly one index $1 \le \ell = \ell(k) \le n$ such that 
$U'_{k,\ell} \equiv 0$ and the remaining $U'_{k,j}$'s are 
i.i.d.~$\mathsf{Unif}(0,m)$.\\
Given the above properties, it is sufficient to take a subsequential 
weak limit of the law of $(\mathbf{U}_k, \mathbf{U}'_k)$ as $k \to \infty$
to obtain a coupling, and the Proposition will be proved.

We start the construction by picking $\mathbf{U}_0$ with the required law,
and letting $\mathbf{U}'_0$ be the identically $0$ vector, so that 
(i) and (ii) hold for $k = 0$.
Then for $k \ge 0$ we recursively define $(\mathbf{U}_{k+1}, \mathbf{U}'_{k+1})$
as follows. Let $\ell = \ell(k)$ be the index such that 
\eqnst
{ U_{k,\ell} 
  < \min \{ U_{k,j} : 1 \le j \le n,\, j \not= \ell \}. }
Let $\ell' \not= \ell$ be any index such that $U'_{k,\ell'} = 0$; where
if there is no such index, we select any $\ell' \not= \ell$ according to a
fixed rule. We then let $U_{k+1,j} = U_{k,j}$ for $j \not= \ell, \ell'$
and $U'_{k+1,j} = U_{k,j}$ for $j \not= \ell, \ell'$. We also let 
$U'_{k+1,\ell} = 0$, and we independently update the triple 
$(U_{k+1,\ell}, U_{k+1,\ell'}, U'_{k+1,\ell'})$ as follows. 
The pair $(U_{k+1,\ell}, U_{k+1,\ell'})$ has the 
conditional distribution of a pair of $\mathsf{Unif}(0,m)$ variables 
$(U_{\ell}, U_{\ell'})$ given $U_{\ell} < U_{\ell'}$ and 
$U_{\ell}^{-\gamma/d} + U_{\ell'}^{-\gamma/d} 
= s - \sum_{j \not= \ell, \ell'} U_{k+1,j}^{-\gamma/d}$,
and $U'_{k+1,\ell'}$ is a $\mathsf{Unif}(0,m)$ variable coupled to 
the above pair in such a way that $U'_{k+1,\ell'} < U_{k+1,\ell'}$
almost surely. This is possible, due to the already proved $n=2$ case.
Then (i) is satisfied for $k+1$, since we have updated the coordinates 
$\ell, \ell'$ according to their conditional law given the other 
coordinates (up to ordering). It is also clear that (ii) is satisfied 
for $k+1$.

It remains to show that (iii) is satisfied. For this, first observe that 
by construction, the number of zero coordinates of $\mathbf{U}'_k$ never 
increases, and once there is only one zero coordinate, this holds for
all larger times. Hence in order to prove (iii), it is sufficient to show
that there are infinitely many times $k$ when $U'_{k,\ell(k)} = 0$.
For this it is sufficient to show that there is $c = c(\gamma/d,s,n,m) > 0$ 
and a finite $K$ such that whatever the value of $\mathbf{U}_k$ is, 
the probability that $U'_{k',\ell(k')} = 0$ for some $k \le k' < k+K$ 
is at least $c$. We break things down according to two (partially overlapping) 
cases the vector $\mathbf{U}_k$ can satisfy. In order to define these cases,
let $\delta = \delta(\gamma/d,s,n,m) > 0$ be sufficiently small with the 
following property:
\eqn{e:property}
{ \text{If $\max_{1 \le i < j \le n} |U_{k,i} - U_{k,j}| < \delta$, then we
have $U_{k,j} < m-\delta$ for all $1 \le j \le n$.} }
Let 
\eqnst
{ I 
  = \left\{ 1 \le j \le n : \left| U_{k,j} - U_{k,\ell(k)} \right| < \delta \right\}. }

\emph{Case (a)} $U_{k,\ell'(k)} < m-\frac{\delta}{4}$.
In this case, there is probability at least 
$c_a = c_a(\gamma/d,\delta) > 0$ that $U_{k+1,\ell'(k)} > U_{k,\ell'(k)}$
and $U_{k+1,\ell(k)} < U_{k,\ell(k)}$. 
On this event, we have
$\ell(k+1) = \ell(k)$, and $U'_{k+1,\ell(k+1)} = U'_{k+1,\ell(k)} = 0$, and
the required property holds.

\emph{Case (b)} $U_{k,\ell'(k)} \ge m - \delta$.\\
Observe (using \eqref{e:property}) that this implies that $\ell'(k) \not\in I$,
and consequently $U_{k,\ell'(k)} > U_{k,\ell(k)} + \delta$.
In this case, there is probability at least $c_b = c_b(\gamma/d,\delta) > 0$
that $U_{k,\ell(k)} < U_{k+1,\ell(k)} < U_{k,\ell(k)} + \frac{\delta}{4}$ and
$U_{k+1,\ell'(k)} < U_{k,\ell'(k)} - \frac{\delta}{4}$.
On this event, the number of indices $j$ such that $U_{k+1,j} \ge m-\delta$
is one less than the corresponding number at time $k$. Hence after at most 
$n-1$ applications of Case (b) we must arrive at Case (a).

This completes the proof of the Proposition.
\end{proof}

\begin{proof}[Proof of Theorem \ref{thm:stoch-mon}]
Take $m = n d / \lambda \omega_{d-1}$ in Proposition \ref{prop:finite-coupling}.
Then $R^{(n)}_j := \left( U^{(j)} \right)^{1/d}$, $j = 1, \dots, n$ has the 
distribution of the (ordered) set of radii of $n$ independent points 
chosen uniformly in a ball of volume $n/\lambda$. 
Let $S^{(n)} := \sum_{i=1}^n \left( R^{(n)}_i \right)^{-\gamma}$.
Proposition \ref{prop:finite-coupling} yields a coupling
between the conditional law of $\{ R^{(n)}_i \}_{i=1}^n$ 
given $S^{(n)} = s$ and the unconditional law of a collection 
$R^{(n)'}_2 < \dots < R^{(n)'}_n$ of $n-1$ radii. 
It remains to pass to the limit $n \to \infty$.

It is sufficient to show that for any $0 < r < \infty$
the conditional distribution (given $S^{(n)} = s$) of the points satisfying 
$R^{(n)}_i < r$ converges to the conditional distribution (given $S = s$)
of the points satisfying $R_i < r$. Let us write $N^{(n)}_r$, respectively
$N_r$ for the number of radii satisfying this property. It is sufficient to 
prove convergence with the value $N^{(n)}_r = k = N_r$ fixed, that is,
to prove the convergence
\eqnspl{e:cond-distr}
{ &\prob \big[ R^{(n)}_1 \le r_1, \dots, R^{(n)}_{k} \le r_{k},
     \text{$R^{(n)}_{j} > r$ for $k+1 \le j \le n$} \,\big|\, S^{(n)} = s \big] \\
  &\qquad \stackrel{n \to \infty}{\longrightarrow}
     \prob \big[ R_1 \le r_1, \dots, R_{k} \le r_{k},
     \text{$R_{j} > r$ for $k+1 \le j < \infty$} \,\big|\, S = s \big] }
for each $r_1 \le \dots \le r_{k} \le r$ and $k \ge 0$. 
We split $S^{(n)}$ into the contributions
\eqnsplst
{ \oS^{(n)}_{k}
  = \sum_{j=k+1}^n (R_j^{(n)})^{-\gamma} \qquad \text{ and } \qquad
  S^{(n)}_{k}
  = \sum_{j=1}^{k} (R_j^{(n)})^{-\gamma}, }
and likewise we split $S$ into the contributions
\eqnsplst
{ \oS_{k}
  = \sum_{j=k+1}^\infty R_j^{-\gamma} \qquad \text{ and } \qquad
  S_{k}
  = \sum_{j=1}^{k} R_j^{-\gamma}, }
We can write the left hand side of \eqref{e:cond-distr} as
\eqnspl{e:conv-decomp}
{ &\frac{\prob [ N^{(n)}_r = k ]}{f_{S^{(n)}}(s)} \, 
     \int f_{\oS^{(n)}_{k} \,|\, N^{(n)}_r}(\os \,|\, k) \,
     f_{S^{(n)}_{k} \,|\, N^{(n)}_r} ( s - \os \,|\, k ) \\
  &\qquad\qquad \times \prob [ R^{(n)}_1 \le r_1, \dots, R^{(n)}_{k} \le r_{k} \,|\, 
  N^{(n)}_r = k,\, S^{(n)}_{k} = s - \os ] \, d\os. }
We claim that the conditional law of $\oS^{(n)}_{k}$ given $N^{(n)}_r = k$
converges to the conditional law of $\oS_{k}$ given $N_r = k$. 
An application of Lemma \ref{lem:large-dev-bnd} yields that 
given $\eps > 0$ we can find $r' > r$ large enough such that 
uniformly in $n$ and $k$ we have
\eqnsplst
{ \prob \Bigg[ \sum_{j : R^{(n)}_j \ge r'} ( R^{(n)}_j )^{-\gamma} \ge \eps 
    \,\Bigg|\, 
     N^{(n)}_r = k \Bigg]
  &< \eps, \\ 
  \prob \Bigg[ \sum_{j : R_j \ge r'} ( R_j )^{-\gamma} \ge \eps 
     \,\Bigg|\, 
     N_r = k \Bigg]
  &< \eps. }
Therefore, the claim follows from the convergence in distribution of the 
radii falling between $r$ and $r'$, which holds due to the relationship between 
$m$ and $n$ (and a binomial to Poisson convergence).
Observe that the limiting law of $\oS_{k}$ given $N_r = k$ 
is continuous.

We turn to the remaining quantities in \eqref{e:conv-decomp}. 
We have $\lim_{n \to \infty} \prob [ N^{(n)}_r = k ] = \prob [ N_r = k ]$ 
again by the choice of $m$. When $k = 0$, this proves the convergence 
sought in \eqref{e:cond-distr}. Henceforth assume $k \ge 1$.

The conditional distribution function of $R^{(n)}_1, \dots, R^{(n)}_{k}$ 
inside the integral in \eqref{e:conv-decomp} is in fact independent of $n$, 
and equals
\eqnst
{ \prob [ R_1 \le r_1, \dots, R_{k} \le r_{k} \,|\, 
  N_r = k,\, S_{k} = s - \os ]. }
This expression is a bounded function of $\os$. 
When $k = 1$, it only takes the values $0$ and
$1$, and has at most one jump in $\os$. When $k \ge 2$, it is continuous 
in $\os$ whenever $f_{S_{k} \,|\, N_r} (s - \os \,|\, k ) > 0$.
The conditional density $f_{S^{(n)}_{k} \,|\, N^{(n)}_r}$ 
inside the integral in \eqref{e:conv-decomp} is independent of $n$, and equals
$f_{S_{k} \,|\, N_r} (s - \os \,|\, k )$. It is also continuous in $\os$ 
when $k \ge 2$. When $k = 1$, it is continuous apart 
from one point. It is also bounded, since it can be written as the $k$-fold 
convolution of the $k = 1$ case, which has a bounded density.
Hence the integral in \eqref{e:conv-decomp} converges to the claimed
limit. Due to Lemma \ref{lem:density-conv} the density 
$f_{S^{(n)}}(s)$ also converges to $f_S(s)$, and
this completes the proof of the Theorem.
\end{proof}

\section{Contribution of up to three points when $\gamma = 4$ and $d = 2$}
\label{sec:integration}

Conditional on there being $1$, $2$, or $3$ Poisson points in an annulus, the
distribution of their contribution to $S$ can be computed. 
In this section we focus on the distribution of the contribution of the 
\emph{nearest} three points.

We condition on $R_1 = r_1$ and $R_4 = r_4$ for some $0 < r_1 < r_4 < \infty$.
We are interested in the conditional distribution of $W = S_2 + S_3
= R_2^{-4} + R_3^{-4}$. The conditional density of $(S_2,S_3)$ takes the 
form:
\eqnst
{ f_{S_2,S_3 | R_1, R_4}(s_2, s_3 | r_1, r_4)
  = \frac{1}{Z_{r_1,r_4}} \frac{1}{s_2^{3/2} s_3^{3/2}}, 
  \qquad s_4 := r_4^{-4} < s_3 < s_2 < r_1^{-4} =: s_1, }
where $Z_{r_1,r_4}$ is a normalization factor. The possible values of 
$W$ are in the interval $(2 s_4, 2 s_1)$. For 
$w \in (2 s_4, 2 s_1)$, we integrate over the range
$w/2 < s_2 < w - s_4$ to find
\eqnsplst
{ f_{W | R_1, R_4} (w | r_1, r_4)
  &= \frac{1}{Z_{r_1,r_4}} \int_{\frac{w}{2}}^{w - s_4}
    \frac{1}{s_2^{-3/2}} \, \frac{1}{(w - s_2)^{-3/2}} \, ds_2 \\
  &= \frac{1}{Z_{r_1,r_4}} \, \frac{1}{w^3} 
    \int_{\frac{w}{2}}^{w - s_4}
    \frac{1}{(\frac{s_2}{w})^{3/2}} \, 
    \frac{1}{(1 - \frac{s_2}{w})^{3/2}} \, ds_2. } 
Changing variables according to $\frac{s_2}{w} = \sin^2 \theta$, this equals
\eqnsplst
{ &f_{W | R_1, R_4} (w | r_1, r_4)
  = \frac{1}{Z_{r_1,r_4}} \, \frac{1}{w^2} 
    \int_{\frac{w}{2}}^{w - s_4}
    \frac{2 \, \sin \theta \, \cos \theta}{\sin^{3} \theta \, \cos^3 \theta}  
    \, d\theta
  = \frac{1}{Z_{r_1,r_4}} \, \frac{8}{w^2} 
    \int_{s_2 = \frac{w}{2}}^{s_2 = w - s_4}
    \frac{1}{\sin^2 2 \theta}  
    \, d\theta \\
  &\qquad = \frac{1}{Z_{r_1,r_4}} \, \frac{4}{w^2} 
    \left[ - \frac{\cos 2 \theta}{\sin 2 \theta}  
    \right]_{s_2 = \frac{w}{2}}^{s_2 = w - s_4} 
  = \frac{1}{Z_{r_1,r_4}} \, \frac{4}{w^2} 
    \left[ \frac{2 \sin^2 \theta - 1}{2 \sin \theta \, \cos \theta}  
    \right]_{s_2 = \frac{w}{2}}^{s_2 = w - s_4} \\
  &\qquad = \frac{1}{Z_{r_1,r_4}} \, \frac{4}{w^2} 
    \frac{2 ( 1 - \frac{s_4}{w} ) - 1}{2 \, \sqrt{1 - \frac{s_4}{w}} 
    \, \sqrt{\frac{s_4}{w}}}
  = \frac{1}{Z_{r_1,r_4}} \, \frac{4}{w^2} 
    \frac{( \frac{1}{2} - \frac{s_4}{w} )}{\sqrt{\frac{5}{4} - 
    (\frac{1}{2} - \frac{s_4}{w})^2}}. } 
Suppose now that we instead condition on $R_1 = r_1$ and $R_5 = r_5$, and 
are interested in the conditional distribution of $Z = S_2 + S_3 + S_4$. 
The range of possible values of $Z$ is $(3 s_5,3 s_1)$, where 
$s_1 = r_1^{-4}$ and $s_5 = r_5^{-4}$. Using the result of the previous 
calculation, for the conditional density of $Z$ we obtain the 
elliptic integral:
\eqnsplst
{ f_{Z | R_1, R_5} (z | r_1, r_5)
  &= \frac{1}{Z_{r_1,r_5}} \int_{z - s_1}^{\frac{2}{3} z} 
    \frac{1}{(z - w)^{3/2}} \,
    \frac{4}{w^2} 
    \frac{( \frac{w}{2} - s_5 )}{\sqrt{s_5 w - s_5^2}} \, dw. }

\appendix

\section{Appendix}
\label{a:calc}

\begin{proof}[Proof of Lemma \ref{lem:density-conv}.]
We have
\eqnst
{ S^{(n)}
  \stackrel{\mathrm{d}}{=} \sum_{j=1}^n \left( U_j \right)^{-\gamma/2}, } 
where $U_1, \dots, U_n$ are i.i.d.~$\Unif(0, n / \lambda \pi)$. Hence
\eqnst
{ \chi_{S^{(n)}}(t)
  := \E \left[ e^{i t S^{(n)}} \right]
  = \chi_W \left( \frac{t}{n^{\gamma/2}} \right)^n, }
where $W = U^{-\gamma/2}$ with $U \sim \Unif(0, 1/\pi \lambda)$. We have
\eqnspl{e:phi(t)}
{ \chi_W(t) 
  &= \E \left[ e^{it U^{-\gamma/2}} \right] 
  = (\pi \lambda) \int_0^{1/\pi \lambda} e^{i t u^{-\gamma/2}} \, du
  = \frac{2 \pi \lambda}{\gamma} \, |t|^{2/\gamma} \, \int_{|t| (\pi\lambda)^{\gamma/2}}^\infty 
    e^{\sgn(t) i v} \, 
    \frac{dv}{v^{\frac{2}{\gamma} + 1}} \\
  &= 1 + \frac{2 \pi \lambda}{\gamma} \, |t|^{2/\gamma} \, \int_{|t| (\pi\lambda)^{\gamma/2}}^\infty 
    \left( e^{\sgn(t) i v} - 1 \right) \, 
    \frac{dv}{v^{\frac{2}{\gamma} + 1}}. }
This gives 
\eqnst
{ \chi_{S^{(n)}}(t)
  = \chi_W \left( \frac{t}{n^{\gamma/2}} \right)^n
  = \left[ 1 + \frac{1}{n} \, \frac{2 \pi \lambda}{\gamma} \, |t|^{2/\gamma} \, 
    \int_{|t| (\pi \lambda)^{\gamma/2} / n^{2/\gamma}}^\infty 
    \left( e^{\sgn(t) i v} - 1 \right) \,  
    \frac{dv}{v^{1 + \frac{2}{\gamma}}} \right]^n. } 
The characteristic function of $S$ on the other hand is 
\eqnspl{e:chiS}
{ \chi_S(t)
  &:= \E \left[ e^{i t S} \right]
  = \exp \left( \frac{2 \pi \lambda}{\gamma} \, |t|^{2/\gamma} \, 
    \int_0^\infty \left( e^{\sgn(t) i v} - 1 \right) \, 
    \frac{dv}{v^{1 + \frac{2}{\gamma}}} \right) \\
  &= \exp \left( - b_2(\gamma) \, |t|^{2/\gamma} \, 
    (1 + i \, \sgn(t) \, \tan(\pi/\gamma) ) \right), }
where $b_2(\gamma) = b_1(\gamma) \, \cos ( \pi/ \gamma ) > 0$.

%
We estimate the $L^1$-distance between $\chi_{S^{(n)}}$ and $\chi_S$.
For $|t| < (\pi \lambda)^{-\gamma/2}$ we expand
\eqnspl{e:phi-expansion}
{ \chi_W(t)
  &= 1 + \frac{2 \pi \lambda}{\gamma} \, |t|^{2/\gamma} 
    \left[ \int_{|t| (\pi \lambda)^{\gamma/2}}^1 
    \left( e^{\sgn(t) i v} -1 \right) \, 
    \frac{dv}{v^{1 + \frac{2}{\gamma}}}
    + \int_1^\infty \left( e^{\sgn(t) i v} -1 \right) \, 
    \frac{dv}{v^{1 + \frac{2}{\gamma}}} \right] \\
  &= 1 - b_2(\gamma) \, |t|^{2/\gamma} \, (1 + i \, \sgn(t) \, \tan(\pi/\gamma) )
    - i \, d \, t + O(t^2), }
where $d$ is a constant. This implies that for $\eps > 0$ small, taking 
$t = \pm \eps n$ we have
\eqnspl{e:phi-calc}
{ &\chi_W \left( \frac{t}{n^{\gamma/2}} \right)^n \\
  &\qquad = \left( 1 - b_2(\gamma) ( 1 + i \, \sgn(t) \, \tan(\pi/\gamma) ) \, 
     \frac{\eps^{2/\gamma}}{n^{1 - 2/\gamma}}
    - i \, d \, \frac{\eps}{n^{\frac{\gamma}{2} - 1}} 
    + O( \frac{\eps^2}{n^{\gamma - 2}} ) \right)^{n^{1 - \frac{2}{\gamma}} 
    \cdot n^{\frac{2}{\gamma}}} \\
  &\qquad = \left[ \exp \left( - b_2(\gamma) \, (1 + i \, \sgn(t) \, \tan(\pi/\gamma) \, 
    \eps^{2/\gamma} \right) \, 
    \left( 1 + O \left( \frac{\eps}{n^{\frac{\gamma}{2} + \frac{2}{\gamma} - 2}} 
    \right) \right) \right]^{n^{\frac{2}{\gamma}}} \\
  &\qquad = \exp \left( - b_2(\gamma) \, (1 + i \, \sgn(t) \, \tan(\pi/\gamma)) \, 
    |t|^{2/\gamma} \right) \, 
    \left( 1 + O \left( \frac{\eps}{n^{\frac{\gamma}{2} - 2}} \right) \right). }
Let us now take $0 \le \eps \le n^{\eta - 1}$, where 
$0 < 2 \eta < \min \{ \frac{\gamma}{2} - 1, 2 \}$. This gives
\eqnspl{e:interval1}
{ \int_{-n^\eta}^{n^{\eta}} 
    \left| \chi_W \left( \frac{t}{n^{\gamma/2}} \right)^n - \chi_S(t) \right| \, dt
  \le O(1) \, n^\eta \, \frac{n^{\eta-1}}{n^{\frac{\gamma}{2} - 2}}
  = O(1) \, n^{2 \eta - \frac{\gamma}{2} + 1}
  \to 0, \quad \text{as $n \to \infty$.} }
Fix now a small $\eps_0 > 0$. When $n^\eta \le |t| \le \eps_0 n^{\gamma/2}$, 
from the expansion \eqref{e:phi-expansion} we have
\eqnst
{ \chi_W \left( \frac{t}{n^{\gamma/2}} \right)
  = \exp \left( - b_2(\gamma) \, |t|^{2/\gamma} \, \frac{1}{n} \, 
    ( 1 + i \, \sgn(t) \, \tan(\pi/\gamma) ) 
    + O ( |t|^{4/\gamma} \, n^{-2} ) \right), }
and hence when $\eps_0$ is small enough, we have
\eqnspl{e:estimate1}
{ \Re \log \chi_W \left( \frac{t}{n^{\gamma/2}} \right)
  &\le - (b_2(\gamma)/2) \, |t|^{2/\gamma} \, \frac{1}{n} \\
  \left| \chi_{S^{(n)}}(t) \right|
  = \left| \chi_W \left( \frac{t}{n^{\gamma/2}} \right)^n \right|
  &\le \exp \left( - (b_2(\gamma)/2) |t|^{2/\gamma} \right). }
In addition (for the same reason),
\eqn{e:estimate2}
{ \left| \chi_S(t) \right|
  \le \exp \left( - (b_2(\gamma)/2) |t|^{2/\gamma} \right). }
The estimates \eqref{e:estimate1} and \eqref{e:estimate2} combine to give
\eqn{e:interval2}
{ \int_{n^{\eta} \le |t| \le {\eps_0 n^{\gamma/2}}}
    \left| \chi_{S^{(n)}}(t) - \chi_{S}(t) \right| \, dt
  \le \exp \left( - (b_2(\gamma)/4) n^{2 \eta / \gamma} \right). }

Finally, in the interval $\eps_0 n^{\gamma/2} \le |t| < \infty$ we have the 
following estimate. From \eqref{e:phi(t)} we have that for some 
$\eps_0 < t_0 < \infty$ and $C > 0$ and for all $|t| \ge t_0$ we have
\eqnst
{ | \chi_W(t) | 
  \le C \, |t|^{-2/\gamma} \le 1/2. }
This implies that 
\eqn{e:interval3}
{ \int_{|t| > t_0 n^{\gamma/2}} \left| \chi_{S^{(n)}}(t) - \chi_{S}(t) \right|
  \le e^{- c n}. }
On the interval $\eps_0 \le |t| \le n^{\gamma/2}$ we have that $\chi_W(t)$ 
is bounded away from $1$ since $U^{-\gamma/2}$ is not a lattice distribution.
This implies 
\eqn{e:interval4}
{ \int_{\eps_0 n^{\gamma/2} \le |t| \le t_0 n^{\gamma/2}}
     \left| \chi_{S^{(n)}}(t) - \chi_{S}(t) \right|
  \le e^{- c n}. }
Putting together the estimates \eqref{e:interval1}, \eqref{e:interval2},
\eqref{e:interval3}, \eqref{e:interval4} gives
\eqnst
{ \sup_{s \in \R} \left| f_{S^{(n)}}(s) - f_{S}(s) \right|
  \le \frac{1}{2 \pi} \int_{\R} \left| \chi_{S^{(n)}}(t) - \chi_{S}(t) \right| \, dt
  \le C \, n^{-\delta} }
with some $\delta = \delta(\gamma) > 0$.
This completes the proof of the Lemma.
\end{proof}

\section*{Acknowledgements.} We thank Keith Briggs (BT Research) for 
bringing the Poisson wireless model to our attention, as well as for 
useful comments on some of our early attempts at approximating
the tail contribution.

\end{document}